 \documentclass[final,3p,times]{elsarticle}

\usepackage{graphicx}
\usepackage{amsmath,amssymb,amsfonts}%
\usepackage{mathrsfs}
\usepackage{amsthm}%
\usepackage{epstopdf} 
\usepackage{amssymb}
\usepackage[ruled,linesnumbered,commentsnumbered]{algorithm2e}
\usepackage{subcaption}
\usepackage{booktabs}
\usepackage{multirow}
\usepackage{bm}
\usepackage{float}
\usepackage{cases}
\usepackage{color}
\usepackage{paracol}

\newtheorem{theorem}{Theorem}[section]
\newtheorem{proposition}{Proposition}[section]%
\newtheorem{assumption}{Assumption}[section]
\newtheorem{lemma}{Lemma}[section]
\newtheorem{remark}{Remark}%

\newtheorem{definition}{Definition}%

\newcommand{\Rnp}{{\mathbb{R}^{n\times p}}}
\newcommand{\F}{{\mathcal{F}}}
\newcommand{\tp}{^{\top}}

\begin{document}

\begin{frontmatter}



\title{Improved Penalty Function Approaches for Optimization  Problems with General Orthogonality}
\author[label1]{Yongshen Zhang}
\author[label2]{Xin Liu}
\author[label3]{Nachuan Xiao}
\author[label1]{Chunming Tang}

\affiliation[label1]{organization={School of Mathematics $ \& $ Center for Applied
Mathematics of Guangxi},
            addressline={Guangxi University},
            city={Nanning},
            postcode={530004},
            country={P.R. China}}
\affiliation[label2]{organization={State Key Laboratory of Scientific and Engineering Computing, Academy of Mathematics and Systems Science},
            addressline={Chinese Academy of Sciences, and University of Chinese Academy of Sciences},
            country={P.R. China}}
\affiliation[label3]{organization={School of Data Science},
            addressline={The Chinese University of Hong Kong},
            city={Shenzhen, Guangdong},
            postcode={518172},
            country={P.R. China}}



\begin{abstract}
In this paper, we consider a class of generalized orthogonal optimization constraint problems (GOOCP) over $\Rnp$, where the variable $X$ is restricted within the intersection of a certain subspace $\F$ and satisfies the quadratic constraint $\{X \in \Rnp: X\tp \phi(X) = I_p\}$. Such constraints generalize a wide range of structured matrix manifolds, such as the Stiefel manifold, the symplectic Stiefel manifold, the indefinite Stiefel manifold, the third-order tensor Stiefel manifold, etc. We show that the feasible region of GOOCP is a closed embedded submanifold of $\Rnp$ and characterize the necessary geometric materials for the existing Riemannian optimization frameworks. Based on the constraint dissolving approach for Riemannian optimization problems, we propose the constraint dissolving penalty function (GOCDF) for the constrained optimization problem GOOCP with easy-to-compute formulations. We further establish the equivalence between GOCDF and GOOCP in the aspects of first-order and second-order stationary points. We also analyze the computational complexity of applying first-order methods to minimize GOOCP, which could be significantly lower than those of first-order Riemannian optimization methods. Numerical experiments demonstrate that solving GOOCP through applying unconstrained optimization methods to minimize constraint dissolving function demonstrates superior efficiency to existing Riemannian optimization methods.

\end{abstract}



\end{frontmatter}




\section{Introduction}
\label{intro}
In this paper, we consider the following generalized orthogonal optimization constraint problem (GOOCP),
\begin{equation}\tag{GOOCP}\label{ocp}
  \begin{aligned}
  &\min_{X\in\mathcal{F}}&&f(X)\\
  &\mathrm{s.t.}&&X^\top\phi(X)=I_p,
  \end{aligned}
\end{equation}
where $ \mathcal{F} $ is a certain subspace of $ \mathbb{R}^{n \times p}$ with $p\leq n$, and $ \phi : \mathcal{F} \rightarrow \mathcal{F} $ is a linear mapping. Throughout this paper, we make the following assumptions on \ref{ocp},
\begin{assumption}\label{ass_1}
  \begin{enumerate}
    \item For any $ X \in \mathcal{F} $, and any $ T \in \mathcal{G} : = \operatorname{span}\{ X_1^{\top}X_2 : X_1, X_2 \in \mathcal{F} \} $, it holds that $ XT \in \mathcal{F} $.
    \item The linear mapping $ \phi $ is one-to-one and self-adjoint.
    \item There exists a linear transform $ \psi : \mathcal{G}\to\mathcal{G} $ such that $ \phi(XT)=\phi(X)\psi(T) $ and  $ \psi((\phi(X))^{\top}X) = X^{\top}\phi(X) $ {holds for any $X \in \F$ and any $T \in \mathcal{G}$}.
  \end{enumerate}
\end{assumption}

Optimization problems in the form of \ref{ocp} have broad applications in the areas of scientific computing \cite{gao2022orthogonalization,liu2015analysis}, statistics \cite{fischler1981random,zou2006sparse}, machine learning \cite{huang2018orthogonal,mackey2018orthogonal}, and signal processing \cite{tropp2005designing}. In the following, we present several illustrative examples of the generality of \ref{ocp}, which enables it to cover a wide range of manifold optimization problems of practical interest. These examples are summarized in Table \ref{tab:ex}. The reader is referred to the monographs \cite{absil2008optimization,boumal2023introduction} and the references therein for detailed real-world applications of these manifold optimization problems.



\begin{table}[H]
    \centering
    \resizebox{\textwidth}{!}{
    \begin{tabular}{l|l|l|l}
    \toprule
       Manifolds & Mathematical expression & $ \mathcal{F} $ & $\phi(X)$ \\
       \midrule
       Stiefel manifold  & ${\rm{St}}(p,n) :=\{ X \in \mathbb{R}^{n \times p} : X^{\top}X = I_p\}$ & $ \Rnp $ & $ X $\\
       \midrule
       Generalized Stiefel manifold (\cite{sato2019cholesky,shustin2023riemannian}) & ${\rm{St}}_M(p,n) :=\{ X \in \mathbb{R}^{n \times p} : X^{\top}BX = I_p,B\succ 0\}$& $ \mathbb{R}^{n \times p} $ & $ BX $ \\
       \midrule
       Symplectic Stiefel manifold (\cite{gao2021riemannian, gao2021geometry}) & ${\rm{Sp}} (2p,2n) :=\{ X \in \mathbb{R}^{2n \times 2p} : X^{\top}J_{2n}X = J_{2p}, J_{2n} = \left(\begin{matrix}
           0 & I_n \\
           -I_n & 0
       \end{matrix} \right) \}$ & $ \mathbb{R}^{2n \times 2p} $ & $ -J_{2n} X J_{2p} $ \\
       \midrule
       Indefinite Stiefel manifold (\cite{van2024riemannian}) & ${\rm iSt}({p,n}) := \{ X \in \mathbb{R}^{n \times p} : X^{\top}AX = J ,A^{\top}=A,J^2 = I_p\}$ & $ \Rnp $ & $ AXJ $ \\
       \midrule
       Hyperbolic manifold (\cite{bai2014minimization}) & $\mathcal{H}(p,n):=\{ X \in \mathbb{R}^{n \times p} : X^{\top}HX = I_p,\lambda(H) = \pm 1\}$ & $ \Rnp $ & $ HX $ \\
       \midrule
       Third-order tensor Stiefel manifold (\cite{mao2024computation}) & ${\rm St}(n,p,l) := \left\{ \mathscr{X} \in \mathbb{R}^{n \times p \times l} : \mathscr{X}^{\top}*\mathscr{X} = \mathscr{I}_p, n \geq p \right\}$ & $ \{ Y \in \mathbb{R}^{ln \times lp} : Y = \operatorname{Diag}(\mathscr{X}\times_3 M)\}  $ & $ X $ \\
       \midrule
    \end{tabular}}
    \caption{Summary of manifolds that satisfy the feasible region of \ref{ocp}}
    \label{tab:ex}
\end{table}

\begin{remark}
    The third-order tensor Stiefel manifold defined via the t-product has been previously introduced in \cite{mao2024computation}. In this paper, we consider a more general setting based on the tensor-tensor product ($l$-product) \cite{kernfeld2015tensor}. Specifically, the tensor-tensor product of $ \mathscr{X} \in \mathbb{R}^{n \times p \times l} $ and $ \mathscr{Y} \in \mathbb{R}^{p \times n \times l} $ defined through the invertible matrix $ M $ is
\begin{equation*}
\mathscr{X} * \mathscr{Y} = \left((\mathscr{X} \times_3 M) \bigtriangleup (\mathscr{Y} \times_3 M) \right)\times_3 M^{-1},
\end{equation*}
where $ (\mathscr{X} \times_3 M)_{i_1i_2{j}} = \sum_{i_3 = 1}^{l} x_{i_1i_2{i_3}}m_{{j}{i_3}} $ and $\bigtriangleup$ is the face-wise product.
Using the operator $ \operatorname{Diag}(\cdot)$ \footnote{For any third-order tensor $ \mathscr{X} \in \mathbb{R}^{n \times p \times l} $, let $ \mathscr{X}^{(i)} $ denote the ith frontal slice of $ \mathscr{X} $. The operator $\operatorname{Diag}(\cdot) $ is defined as $ \operatorname{Diag}(\mathscr{X}) = \left[\begin{matrix}
                    {\mathscr{X}}^{(1)} & 0 & 0 \\
                    0 & \ddots & 0 \\
                    0 & 0 & {\mathscr{X}}^{(l)}
                  \end{matrix} \right] $.}, the third-order tensor Stiefel manifold can be reformulated as
\begin{equation*}
  {\rm St}(n,p,l) = \left\{ \mathscr{X} \in \mathbb{R}^{n \times p \times l} : \operatorname{Diag}(\mathscr{X}\times_3 M)^{\top}\operatorname{Diag}(\mathscr{X}\times_3 M) = I_{lp} \right\}.
\end{equation*}
Let $ \mathcal{F} = \{ Y \in \mathbb{R}^{ln \times lp} : Y = \operatorname{Diag}(\mathscr{X}\times_3 M)\} $ and $ \phi(X) = X $, then the feasible region of \ref{ocp} is equivalent to the third-order tensor Stiefel manifold.
\end{remark}

\subsection{Related Works}
As demonstrated in Table \ref{tab:ex}, the feasible region of the applications of \ref{ocp} usually exhibits a specific manifold structure. As a result, substantial research has been devoted to solving these special cases of \ref{ocp} as an unconstrained optimization problem on a Riemannian manifold \cite{absil2008optimization}. Due to the nonlinear structure of the manifold, Riemannian optimization methods require the use of fundamental geometric tools (namely concepts from differential geometry), such as the geodesic, parallel transport, etc. The geodesic serves as the generalization of straight lines in Euclidean space to the Riemannian manifold. However, computing geodesics on a manifold $\mathcal{M}$ is generally computationally expensive, even when $\phi$ is given by relatively simple expressions (e.g., the Stiefel manifold). Although the notion of retraction was introduced in \cite{absil2008optimization} to approximate geodesic, retraction operations often remain significantly more costly than standard matrix-matrix multiplications in various scenarios, including the Stiefel manifold, the symplectic Stiefel manifold, and the third-order tensor Stiefel manifold. Additionally, parallel transport moves tangent vectors between different tangent spaces on a manifold. However, computing the parallel transport typically requires solving an ordinary differential equation, making it computationally expensive in practice. To mitigate this issue, the concept of vector transport was introduced in \cite{absil2008optimization} as an approximation to parallel transport. As mentioned in \cite{qi2010riemannian}, vector transport is generally much more computationally efficient. With retractions and vector transports, many unconstrained optimization approaches have been extended to their Riemannian versions, including Riemannian gradient descent with line-search \cite{abrudan2008steepest,absil2008optimization,gao2021riemannian,van2024riemannian}, Riemannian conjugate gradient methods \cite{sato2022riemannian,tang2023class,zhu2017riemannian}, Riemannian accelerated gradient methods \cite{han2023riemannian,liu2017accelerated,siegel2019accelerated,zhang2018towards}, see \cite{absil2008optimization,boumal2023introduction} for instances.  Table \ref{tab:man} summarizes existing work on the computational cost of these geometric materials. Furthermore, in existing Riemannian optimization approaches for solving \ref{ocp}, the geometric structures of $\mathcal{M}$ is characterized separately. To the best of our knowledge, there has been no work establishing an uniform characterization of the geometric structures for \ref{ocp}.

Furthermore, based on these geometric tools, converting an unconstrained optimization method into its Riemannian version for solving \ref{ocp} requires substantial modifications. As a result, keeping Riemannian optimization methods aligned with the rapid advances in unconstrained nonconvex optimization remains challenging. Moreover, available Riemannian optimization solvers are considerably fewer than their unconstrained counterparts.

\begin{table}[!htbp]
    \centering
    \caption{A summary of geometric tools for specific manifolds that satisfy the \ref{ocp} constraints. For all entries in the fourth column, the orthogonal projection is selected as the vector transport. ``E'' and ``I'' denote explicit and implicit forms, respectively. When the orthogonal projection is implicit, the computational complexity is determined by the Python solver \emph{scipy.linalg.solve\_continuous\_lyapunov}.}
    \label{tab:man}
    \resizebox{0.75\textwidth}{!}{
    \begin{tabular}{l|l|c|c|c}
    \toprule
       Manifolds & Retraction & Retraction cost & Vector transport & Vector transport cost \\
       \midrule
       ${\rm{St}}(p,n)$ & QR decomposition & $ \mathcal{O}(np^2) $ & E & $ \mathcal{O}(np^2) $\\
       \midrule
       ${\rm{St}}_M(p,n)$ & QR decomposition & $ \mathcal{O}(np^2) $ & E & $ \mathcal{O}(n^2p+np^2) $ \\
       \midrule
       ${\rm{Sp}} (2p,2n)$ & Cayley retraction & $  \mathcal{O}(n^3)  $ & I & $ \mathcal{O}(n^2p+p^3) $ \\
       \midrule
       ${\rm iSt}(p,n)$ & Cayley retraction & $  \mathcal{O}(n^3)  $ & I & $ \mathcal{O}(n^2p+p^3) $ \\
       \midrule
       ${\rm St}(n,p,l)$ & t-QR decomposition & $ \mathcal{O}(np^2l+npl^2) $ & E & $ \mathcal{O}(np^2l+npl^2) $ \\
       \bottomrule
    \end{tabular}}
\end{table}


Recently, \cite{xiao2024dissolving} introduced an exact penalty function for solving optimization problems with Riemannian constraints in Euclidean space, referred to as the constraint dissolving approach. This approach effectively bridges the gap between Euclidean unconstrained optimization methods and Riemannian constrained optimization problems, while preserving the desirable theoretical properties of existing Euclidean approaches. The central part of their method is the constraint dissolving operator, whose general formulation requires the Jacobian matrix of the constraint mapping $c$ and its pseudo-inverse. Therefore, computing the constraint dissolving mapping for \ref{ocp} based on \cite{xiao2024dissolving} can be computationally expensive for a wide range of manifolds. Subsequently, Jiang et al. \cite{jiang2026smooth} developed a smooth locally exact penalty method for optimization problems on generalized Stiefel manifolds. This method can be regarded as a further extension of the constraint dissolving framework to a specific class of manifolds. In their work, a computationally tractable constraint dissolving operator was constructed based on the closed-form expression of the Lagrange multiplier associated with the generalized Stiefel manifold. However, the \ref{ocp} does not admit such a closed-form expression of the Lagrange multiplier. As a result, the technique proposed in \cite{jiang2026smooth} cannot be directly extended to the present problem. The unified easy-to-compute formulation of the constraint dissolving for $\mathcal{M}$ remain unknown yet.


\subsection{Contributions}
In this paper, we present an uniform characterization of the geomeric materials of $\mathcal{M}$, including the tangent space, Riemannian gradient and Riemannian Hessian.
Moreover, we construct a computationally efficient constraint dissolving operator and subsequently propose a constraint dissolving function for $\mathcal{M}$ in \ref{ocp}. Then we propose a general and extensible algorithmic framework for solving \ref{ocp}, which eliminates the need for case-specific analysis of manifold structures and avoids reliance on computational complexity geometry tools. Extensive numerical experiments demonstrate significant advantages in formulating \ref{ocp} as a Riemannian optimization problem and implementing efficient unconstrained optimization methods through the constraint dissolving framework.

\subsection{Organization}

The rest of this paper is organized as follows. In Section \ref{sec:2}, the necessary preliminaries are introduced, including notation, terminology, and foundational concepts used throughout the paper.  It is also proved that the feasible region of \ref{ocp} constitutes a closed embedded submanifold in $ \mathcal{F} $, and characterizes uniformly the geometric materials of $ \mathcal{M} $. In Section \ref{sec:3}, a computationally efficient constraint dissolving operator $\mathcal{A}$ is proposed, based on which the constraint dissolving function associated with \ref{ocp} is constructed and the relationship between \ref{cdf} and \ref{ocp} is established. In Section \ref{sec:4}, the fundamental algorithmic framework is described, and the computational complexity and stability of infeasible points are analyzed. In Section \ref{sec:5}, numerical experiments are presented to demonstrate that the constraint reduction framework embedded with unconstrained optimization methods has better numerical performance compared to the Riemannian optimization methods.

\section{Preliminaries}\label{sec:2}
In this section, we first introduce the basic notation used throughout the paper, present the necessary definitions and assumptions, and then analyze the manifold structure of the feasible region of \ref{ocp}.
\subsection{Notations and terminologies}
Our notation and terminology for manifolds follows exactly the standard literature \cite{absil2008optimization} and for tensors \cite{kolda2009tensor}.

Throughout this paper, vectors are written in italic lowercase letters such as $u$, $v$, matrices correspond to uppercase letters, e.g., $A$, $B$, and tensors are denoted by calligraphic capital letters such as $\mathscr{A}$, $\mathscr{B}$. A slice is a two-dimensional section of a tensor, defined by fixing all but two indices. The $k$th frontal slice $ {\mathscr{X}}_{::k} $ of the third-order tensor $ \mathscr{X} \in \mathbb{R}^{n \times  p  \times l} $ is denoted as $ \mathscr{X}^{(k)} $. The manifold is denoted by $\mathcal{M}$. In this paper, $\mathcal{M}$ specifically represents the feasible set of \ref{ocp}, that is, $ \mathcal{M}: = \{ X \in \mathcal{F} : X^{\top}\phi(X) = I \} $. The tangent space and normal space of $\mathcal{M}$ at $X$ are denoted by $T_X \mathcal{M}$ and $N_X \mathcal{M}$, respectively. For a given matrix $A \in \mathbb{R}^{n \times n}$, we denote its trace as $ \operatorname{tr}(A) $. Then, the Euclidean inner product of two matrices $ X, Y \in \mathbb{R}^{n \times k} $ is defined as $ \left\langle X, Y \right\rangle = \operatorname{tr} (X^{\top}Y) $. The Frobenius norm of a matrix $ A $ induced by this inner product is $ \| A \| = \sqrt{\left\langle A, A \right\rangle} $. We set the Riemannian  metric on the manifold $\mathcal{M}$ as the metric inherited from the standard inner product in $ \mathbb{R}^{n \times p} $. For any square matrix $ T $ we define the generalized symmetrization mapping $ \Phi(T) = T^{\top} + \psi(T) $. Furthermore, the range and kernel of $ \Phi $ are defined as $ \mathcal{S}_1: = \{ \Phi(T) : T \in \mathcal{G} \} $ and $ \mathcal{S}_2 := \{ T \in \mathcal{G}: \Phi(T) = 0  \} $, respectively.

Next, we introduce the constraint mapping
\begin{equation}\label{eq:C}
  C(X):= X^{\top}\phi(X) - I_p.
\end{equation}
It can be seen that the feasible region of \ref{ocp} is precisely the kernel of $C$. By comparing the constraint dissolving operators on the Stiefel and the generalized Stiefel manifolds \cite{xiao2024dissolving,jiang2026smooth}, we define the following mapping
\begin{equation}\label{eq:A}
  \mathcal{A}(X) := \frac{3}{2}X - \frac{1}{2}X\phi(X)^{\top}X.
\end{equation}
The corresponding composite mapping is
\begin{equation*}
  (C \circ \mathcal{A})(X) = \frac{9}{4}X^{\top}\phi(X) - \frac{3}{2}(X^{\top}\phi(X))^2 + \frac{1}{4}(X^{\top}\phi(X))^3 - I_p.
\end{equation*}
For simplicity, we define the function $ g(X): \mathcal{F} \to \mathbb{R} $ and the mapping $ G(X) : \mathcal{F} \to \mathcal{F} $ as follows
\begin{equation*}
  g(X) := f(\mathcal{A}(X)), G(X) := \nabla f(\mathcal{A}(X)),
\end{equation*}
respectively.
We now present the Fr\'echet derivatives of the mappings $C$, $ \mathcal{A} $ and $ C \circ \mathcal{A} $ as follows.
\begin{lemma}\label{le1}
  For any $ X, Z \in \mathcal{F} $, $ T \in \mathcal{G} $ it holds that
  \begin{equation}\label{d1}
    {\rm D}C(X)[Z] = X^{\top}\phi(Z) + Z^{\top}\phi(X),
  \end{equation}
  and its adjoint is
  \begin{equation}\label{d2}
    {\rm D}C(X)^*[T] = \phi(X)\Phi(T).
  \end{equation}
\end{lemma}

\begin{lemma}\label{le2}
  For any $ X, Z \in \mathcal{F} $, it holds that
  \begin{equation}\label{d3}
    {\rm D}\mathcal{A}(X)[Z] = \frac{3}{2}Z - \frac{1}{2}\left( Z\left(\phi(X) \right)^{\top}X + X\left(\phi(Z) \right)^{\top}X + X\left(\phi(X) \right)^{\top}Z \right).
  \end{equation}
  Moreover, its adjoint is
  \begin{equation}
    {\rm D}\mathcal{A}(X)^*[T] = T\left(\frac{3}{2}I_p - \frac{1}{2}X^{\top}\phi(X)\right) - \frac{1}{2}\phi(X)\left(  X^{\top}T + \psi(T^{\top}X)  \right).
  \end{equation}
\end{lemma}

\begin{lemma}\label{le3}
  For any $ X, Z \in \mathcal{F} $, it holds that
  \begin{equation}\label{d4}
     \begin{aligned}
            {\rm D}(C \circ \mathcal{A})(X)[Z] & = \frac{9}{4}{\rm D}C(X)[Z] - \frac{3}{2}\left({\rm D}C(X)[Z] X^{\top}\phi(X)+ X^{\top}\phi(X){\rm D}C(X)[Z]\right) \\
             +  & \frac{1}{4}\left( {\rm D}C(X)[Z]\left( X^{\top}\phi(X) \right)^2 + X^{\top}\phi(X){\rm D}C(X)[Z]X^{\top}\phi(X) + \left( X^{\top}\phi(X) \right)^2{\rm D}C(X)[Z] \right).
     \end{aligned}
  \end{equation}
\end{lemma}

\subsection{Constants}
This subsection starts with the following assumptions.
\begin{assumption}\label{ass_2}
  \begin{itemize}
    \item The objective function $ f $ is locally Lipschitz continuous and smooth on $\mathcal{F}$.
    \item The mapping $ \phi $ is Lipschitz continuous on $\mathcal{F}$.
  \end{itemize}
\end{assumption}
Next, we introduce some necessary constants. Let $ \delta_{C,X} $ denote the smallest singular value of ${\rm D}C(X)$ i.e., for any $ X, Z \in \mathcal{F} $, it holds that
\begin{equation*}
    \| {\rm D}C(X)[Z] \| \geq \delta_{C,X}\| Z \|.
\end{equation*}
For any $ X \in \mathcal{M} $, we define the positive scalar $ \rho_X \leq 1 $ as
\begin{align*}
  \rho_X := & {\arg\max}_{0 < \rho \leq 1} \rho\\
  {\rm s.t.} &  \| {\rm D}C(Y)[Z] \| \geq \delta_{C,X}\| Z \|,\\
   & \| Y - X \| \leq \rho, \\
   & Y \in \mathcal{F}.
\end{align*}
We then denote $ \Delta_X := \{ Y \in \mathcal{F}: \| Y-X \| \leq \rho_X \} $. Based on this local set, we introduce the following constants:
\begin{itemize}
    \item $ M_{X,f} := \sup_{Y \in \Delta_X } \| \nabla f(\mathcal{A}(Y)) \| $;
    \item $ L_\phi := \sup_{Y_1,Y_2 \in \mathcal{F}, Y_1 \neq Y_2} \frac{\| \phi(Y_1) - \phi(Y_2)) \|}{\|Y_1 - Y_2\|} $;
    \item $ L_{X,g} := \sup_{Y_1,Y_2 \in \Delta_X, Y_1 \neq Y_2} \frac{\| \nabla f(Y_1) - \nabla f(Y_2)) \|}{\|Y_1 - Y_2\|} $.
\end{itemize}

We then estimate the Lipschitz constants of the differentials of several mappings used in this paper.

\begin{lemma}
  For any $ X \in \mathcal{M}, Z \in \mathcal{F} $, $ Y_1, Y_2 \in \Delta_X $, it holds that
  \begin{equation*}
    \| {\rm D}C(Y_1)[Z] - {\rm D}C(Y_2)[Z] \| \leq 2 L_{\phi}\| Z \|  \| Y_1 - Y_2 \|.
  \end{equation*}
\end{lemma}

\begin{lemma}\label{lem:DA}
  For any $ X \in \mathcal{M}, Z, T \in \mathcal{F} $, $ Y_1, Y_2 \in \Delta_X $, there exists a constant $\alpha_{\mathcal{A}}$ such that
  \begin{equation*}
    \| {\rm D}\mathcal{A}(Y_1)[Z] - {\rm D}\mathcal{A}(Y_2)[Z] \| \leq \alpha_{\mathcal{A}} L_{\phi}\| Z \|  \| Y_1 - Y_2 \|,
  \end{equation*}
  and
  \begin{equation*}
    \| {\rm D}\mathcal{A}(Y_1)^*[T] - {\rm D}\mathcal{A}(Y_2)^*[T] \| \leq \alpha_{\mathcal{A}} L_{\phi}\| T \|  \| Y_1 - Y_2 \|.
  \end{equation*}

\end{lemma}
\begin{proof}
  For any $ Y_1, Y_2 \in \Delta_X $, we calculate that
  \begin{equation*}
    \begin{aligned}
        & \left\| {\rm D}\mathcal{A}(Y_1)[Z] - {\rm D}\mathcal{A}(Y_2)[Z] \right\| \\
        = & \frac{1}{2}\left\| Z\phi(Y_1)^{\top}Y_1 + Y_1\phi(Z)^{\top}Y_1 + Y_1\phi(Y_1)^{\top}Z - Z\phi(Y_2)^{\top}Y_2 - Y_2\phi(Z)^{\top}Y_2 - Y_2\phi(Y_2)^{\top}Z \right\| \\
        \leq & \frac{1}{2}\left( \left\| Z\phi(Y_1)^{\top}Y_1 - Z\phi(Y_2)^{\top}Y_2 \right\| + \left\| Y_1\phi(Z)^{\top}Y_1 - Y_2\phi(Z)^{\top}Y_2 \right\| + \left\| Y_1\phi(Y_1)^{\top}Z - Y_2\phi(Y_2)^{\top}Z \right\| \right).
    \end{aligned}
  \end{equation*}
  We control the three terms on the right hand side of the above inequality separately as
  \begin{equation*}
    \begin{aligned}
        \left\| Z\phi(Y_1)^{\top}Y_1 - Z\phi(Y_2)^{\top}Y_2 \right\| = & \left\| \left( Z\phi(Y_1)^{\top}Y_1 - Z\phi(Y_1)^{\top}Y_2 \right) + \left( Z\phi(Y_1)^{\top}Y_2 - Z\phi(Y_2)^{\top}Y_2\right) \right\| \\
        \leq & \left\|  Z\phi(Y_1)^{\top}Y_1 - Z\phi(Y_1)^{\top}Y_2 \right\| + \left\| Z\phi(Y_1)^{\top}Y_2 - Z\phi(Y_2)^{\top}Y_2 \right\| \\
        \leq & L_\phi \|Z\| \|Y_1\| \| Y_1 - Y_2 \| + L_\phi \|Z\| \|Y_2\| \| Y_1 - Y_2 \|\\
        \leq & L_\phi \|Z\| \left( \|Y_1\| + \|Y_2\| \right) \| Y_1 - Y_2 \|,
    \end{aligned}
  \end{equation*}
  \begin{equation*}
    \begin{aligned}
        \left\| Y_1\phi(Z)^{\top}Y_1 - Y_2\phi(Z)^{\top}Y_2 \right\| = & \left\| \left( Y_1\phi(Z)^{\top}Y_1 - Y_1\phi(Z)^{\top}Y_2 \right) + \left( Y_1\phi(Z)^{\top}Y_2 - Y_2\phi(Z)^{\top}Y_2 \right) \right\| \\
        \leq & \left\| Y_1\phi(Z)^{\top}Y_1 - Y_1\phi(Z)^{\top}Y_2 \right\| + \left\| Y_1\phi(Z)^{\top}Y_2 - Y_2\phi(Z)^{\top}Y_2 \right\| \\
        \leq & L_\phi \|Z\| \|Y_1\| \| Y_1 - Y_2 \| + L_\phi \|Z\| \|Y_2\| \| Y_1 - Y_2 \|\\
        \leq & L_\phi \|Z\| \left( \|Y_1\| + \|Y_2\| \right) \| Y_1 - Y_2 \|,
    \end{aligned}
  \end{equation*}
  and
  \begin{equation*}
    \begin{aligned}
        \left\| Y_1\phi(Y_1)^{\top}Z - Y_2\phi(Y_2)^{\top}Z \right\| = & \left\| \left( Y_1\phi(Y_1)^{\top}Z - Y_1\phi(Y_2)^{\top}Z \right) + \left( Y_1\phi(Y_2)^{\top}Z - Y_2\phi(Y_2)^{\top}Z\right) \right\| \\
        \leq & \left\|  Y_1\phi(Y_1)^{\top}Z - Y_1\phi(Y_2)^{\top}Z \right\| + \left\| Y_1\phi(Y_2)^{\top}Z - Y_2\phi(Y_2)^{\top}Z \right\| \\
        \leq & L_\phi \|Z\| \|Y_1\| \| Y_1 - Y_2 \| + L_\phi \|Z\| \|Y_2\| \| Y_1 - Y_2 \|\\
        \leq & L_\phi \|Z\| \left( \|Y_1\| + \|Y_2\| \right) \| Y_1 - Y_2 \|.
    \end{aligned}
  \end{equation*}
  Combining these three terms, and considering that $ \|Y_1\| $ and $ \|Y_2\| $ are bounded, one obtains that there exists a constant $\alpha_{\mathcal{A}}$ such that
  \begin{equation*}
    \| {\rm D}\mathcal{A}(Y_1)[Z] - {\rm D}\mathcal{A}(Y_2)[Z] \| \leq \alpha_{\mathcal{A}} L_{\phi}\| Z \|  \| Y_1 - Y_2 \|.
  \end{equation*}
  In the same way we can deflate $ \| {\rm D}\mathcal{A}(Y_1)^*[T] - {\rm D}\mathcal{A}(Y_2)^*[T] \| $ and the result is the same as above. Thus the proof is completed.
\end{proof}

The proof of Lemma \ref{lem:DA} shows that for any $ Y \in \Delta_X $, the estimate $ \|Y\| \leq \frac{1}{3}\alpha_{\mathcal{A}} $ holds. Furthermore, we can conclude that
\begin{equation}
    \| {\rm D}\mathcal{A}(Y)[Z]\| \leq (\frac{3}{2} + \frac{1}{6}\alpha_{\mathcal{A}}^2L_\phi)\| Z \|,
\end{equation}
where $ Z \in \mathcal{F} $.

\begin{lemma}
  For any $ X \in \mathcal{M}, Z \in \mathcal{F} $, $ Y_1, Y_2 \in \Delta_X $, there exists a constant $\alpha_{C,\mathcal{A}}$ such that
  \begin{equation*}
    \| {\rm D}(C \circ \mathcal{A})(Y_1)[Z] - {\rm D}(C \circ \mathcal{A})(Y_2)[Z] \| \leq 2 \alpha_{C,\mathcal{A}} L_{\phi}\| Z \| \| Y_1 - Y_2 \|.
  \end{equation*}
\end{lemma}
\begin{proof}
  It follows from Lemma \ref{le3} that
  \begin{equation*}
    {\rm D}(C \circ \mathcal{A})(X)[Z] = \frac{9}{4}T_1(X) - \frac{3}{2}T_2(X) + \frac{1}{4}T_3(X),
  \end{equation*}
  where
  \begin{align*}
    T_1(X) = & X^{\top}\phi(Z) + Z^{\top}\phi(X), \\
    T_2(X) = & (X^{\top}\phi(Z) + Z^{\top}\phi(X))X^{\top}\phi(X) + X^{\top}\phi(X)(X^{\top}\phi(Z) + Z^{\top}\phi(X)),\\
    T_3(X) = & (X^{\top}\phi(Z) + Z^{\top}\phi(X))(X^{\top}\phi(X))^2 + X^{\top}\phi(X)(X^{\top}\phi(Z) + Z^{\top}\phi(X))X^{\top}\phi(X) \\
    & + (X^{\top}\phi(X))^2(X^{\top}\phi(Z) + Z^{\top}\phi(X)).
  \end{align*}
  Notice that
  \begin{equation*}
    \| {\rm D}(C \circ \mathcal{A})(Y_1)[Z] - {\rm D}(C \circ \mathcal{A})(Y_2)[Z] \| \leq \frac{9}{4}\| T_1(Y_1) - T_1(Y_2) \| + \frac{3}{2}\| T_2(Y_1) - T_2(Y_2) \| + \frac{1}{4}\| T_3(Y_1) - T_3(Y_2) \|,
  \end{equation*}
  and
  \begin{equation*}
    \| T_1(Y_1) - T_1(Y_2) \| \leq \left( \| \phi(Z) \| + L_{\phi}\| Z \| \right) \| Y_1 - Y_2 \|.
  \end{equation*}
  We discuss the Lipschitz property of $T_2(X)$ below. We note that $T_2(X)$ can be recharacterized as
  \begin{equation*}
    T_2(X) = T_1(X)X^{\top}\phi(X)+ X^{\top}\phi(X)T_1(X).
  \end{equation*}
  Then we have
  \begin{equation*}
    \begin{aligned}
         T_2(Y_1)-T_2(Y_2)&=T_1(Y_1)Y_1^\top\phi(Y_1)+Y_1^\top\phi(Y_1)T_1(Y_1) - T_1(Y_2)Y_2^\top\phi(Y_2)-Y_2^\top\phi(Y_2)T_1(Y_2) \\
         & = \left[T_1(Y_1)Y_1^\top\phi(Y_1)-T_1(Y_2)Y_2^\top\phi(Y_2)\right] + \left[ Y_1^\top\phi(Y_1)T_1(Y_1)-Y_2^\top\phi(Y_2)T_1(Y_2)\right]
    \end{aligned}
  \end{equation*}
  In the following, we control both $ T_1(Y_1)Y_1^\top\phi(Y_1)-T_1(Y_2)Y_2^\top\phi(Y_2) $ and $ Y_1^\top \phi(Y_1) T_1(Y_1) - Y_2^\top \phi(Y_2) T_1(Y_2) $ separately.
  \begin{equation*}
    \begin{aligned}
        \| T_1(Y_1)Y_1^\top\phi(Y_1)-T_1(Y_2)Y_2^\top\phi(Y_2) \| \leq & \|T_1(Y_1)\|\|Y_1^\top\phi(Y_1)-Y_2^\top\phi(Y_2)\|+\|T_1(Y_1)-T_1(Y_2)\|\|Y_2^\top\phi(Y_2)\| \\
        \leq &\|T_1(Y_1)\|\cdot L_{X\phi}\|Y_1-Y_2\|+\left(\|\phi(Z)\|+L_\phi\|Z\|\right)\|Y_1-Y_2\|\cdot\sigma_{\max},
    \end{aligned}
  \end{equation*}
  where $ L_{X\phi} $ is the Lipschitz constant for $ X^{\top}\phi(X) $ and $\sigma_{\max}$ is the largest eigenvalue of $  Y^{\top}\phi(Y) $ for $ \forall Y \in \Delta_X $. Similarly, we have
  \begin{equation*}
    \| Y_1^\top \phi(Y_1) T_1(Y_1) - Y_2^\top \phi(Y_2) T_1(Y_2) \| \leq \|T_1(Y_1)\|\cdot L_{X\phi}\|Y_1-Y_2\|+\left(\|\phi(Z)\|+L_\phi\|Z\|\right)\|Y_1-Y_2\|\cdot\sigma_{\max}.
  \end{equation*}
  Therefore, we can deduce that
  \begin{equation*}
    \|T_2(Y_1)-T_2(Y_2)\|\leq2\left(\sigma_{\max}\cdot \left( \|\phi(Z)\| + L_\phi\| Z\| \right) + L_{X\phi} \|T_1(Y_1)\| \right) \|Y_1-Y_2\|.
  \end{equation*}
  We find that $ T_3(X) $ is also a polynomial function of $ T_1(X) $ and $ X^{\top}\phi(X) $. Thus $ T_3(X) $ also has a Lipschitz property similar to that of $ T_2(X) $. Further, based on the Lipschitz property of $ T_1(X) $, we get that $ \|T_1(Y_1)\| $ is similarly controllable by $ \|\phi(Z)\|+L_\phi\|Z\| $. Thus we assert that there exists a constant $ \alpha_{C,\mathcal{A}} $ such that
  \begin{equation*}
    \| {\rm D}(C \circ \mathcal{A})(Y_1)[Z] - {\rm D}(C \circ \mathcal{A})(Y_2)[Z] \| \leq \alpha_{C,\mathcal{A}}\left( \| \phi(Z) \| + L_{\phi}\| Z \| \right) \| Y_1 - Y_2 \| \leq 2 \alpha_{C,\mathcal{A}} L_{\phi}\| Z \| \| Y_1 - Y_2 \|.
  \end{equation*}
  The proof is completed.
\end{proof}

In addition, we define
\begin{equation*}
  \varepsilon_X := \min \left\{ \frac{\rho_X}{2}, \frac{\delta_{C,X}}{4L_\phi}, \frac{\delta_{C,X}}{2\alpha_{\mathcal{A}}^2L_\phi^2 + 6L_\phi}, \frac{3\delta_{C,X}^2}{32\alpha_{\mathcal{A}}\alpha_{C,\mathcal{A}}L_\phi^2} \right\},
\end{equation*}
and denote the $\varepsilon_X$ neighborhood of $ X $ as
\begin{equation*}
  \Omega_X := \{ Y \in \mathcal{F}: \| Y-X \| \leq \varepsilon_X \}.
\end{equation*}

\subsection{Riemannian Structure of $ \mathcal{M} $}
In this subsection, the set $ \mathcal{M} $ is shown to be an embedded submanifold of $ \mathcal{F} $ and some important geometric tools are given. These materials allow us to develop manifold optimization methods for solving \ref{ocp}.

\begin{lemma}
  For any $ X \in \mathcal{F} $, it holds that $ C(X) \in \mathcal{S}_1 $.
\end{lemma}
\begin{proof}
  For any $ X \in \mathcal{F} $, there exists $ T = \frac{1}{2}\left( (\phi(X))^{\top}X \right) \in \mathcal{G} $ such that
  \begin{equation*}
    T^{\top} + \psi(T) = X^{\top}\phi(X) = C(X).
  \end{equation*}
  Therefore, we conclude that $ C(X) \in \mathcal{S}_1 $.
\end{proof}


\begin{proposition}
  The set $ \mathcal{M} $ is a closed embedded submanifold of $ \mathcal{F} $ with the dimension $ \operatorname{dim}(\mathcal{M}) = \operatorname{dim}(\mathcal{F}) + \operatorname{dim}(\mathcal{S}_2) - \operatorname{dim}(\mathcal{G})  $.
\end{proposition}
\begin{proof}
  Consider the map $ C(X) = X^{\top}\phi(X) - I $. Combining the continuity of the map $ C $ and $ \mathcal{M} = C^{-1}(0) $, one obtains that $ \mathcal{M} $ is a closed set of $ \mathcal{F} $.

  Let $ X \in \mathcal{M} $. We then prove that $ {\rm D}C(X) $ is a surjection, i.e., for all $ \bar{Z} \in \{ T^{\top} + \psi(T) : T \in \mathcal{G} \} $, there exists $ Z \in \mathcal{F} $ such that $ {\rm D}C(X)[Z] = \bar{Z} $. Let $ \bar{Z} = \bar{T}^{\top} + \psi(\bar{T}) $. Constructing $ Z = X\bar{T} $, one obtains that
  \begin{equation*}
    {\rm D}C(X)[Z] = X^{\top}\phi(X)\psi(\bar{T}) + \bar{T}^{\top}X^{\top}\phi(X) = \bar{Z}.
  \end{equation*}
  Therefore $ C $ is full rank. Furthermore, $ \operatorname{rank}(C)(X) = \operatorname{dim}(\mathcal{S}_1) = \operatorname{dim}(\mathcal{G}) - \operatorname{dim}(\mathcal{S}_2) $ By the submersion theorem \cite{absil2008optimization}, we conclude that $ \mathcal{M} $ is a closed embedded submanifold of $ \mathcal{F} $. Its dimension is $ \operatorname{dim}(\mathcal{M}) = \operatorname{dim}\left(C^{-1}(0)\right) = \operatorname{dim}(\mathcal{F}) + \operatorname{dim}(\mathcal{S}_2) - \operatorname{dim}(\mathcal{G}) $.
\end{proof}

\begin{proposition}
  Given $ X \in \mathcal{M} $, the tangent space of $ \mathcal{M} $ at $ X $ admits the following expressions
  \begin{align}
    T_X\mathcal{M} = & \left\{ Z \in \mathcal{F} : X^{\top}\phi(Z) + Z^{\top}\phi(X) = 0 \right\} \label{tx1}\\
    = & \left\{ Z \in \mathcal{F} : \left\langle Z, \phi(X)(T^{\top} + \psi(T)) \right\rangle = 0 , \forall T \in \mathcal{G} \right\}.
  \end{align}
\end{proposition}
%

\begin{proposition}
  Given $ X \in \mathcal{M} $, the normal space of $ \mathcal{M} $ at $ X $ admits the following expression
  \begin{equation*}
    N_X\mathcal{M} = \left\{ \phi(X)(T^{\top} + \psi(T)): T \in \mathcal{G} \right\}.
  \end{equation*}
\end{proposition}

\begin{proposition}
  For any $ D \in \mathcal{F} $, the projection from $ D $ to $ T_X\mathcal{M} $ can be expressed as
  \begin{equation}\label{pro}
    \mathcal{P}_{T_X\mathcal{M}}(D) = D - \phi(X)\Theta_X(D),
  \end{equation}
  where $ \Theta_X(D) $ is the solution for the following least square problem
  \begin{equation}\label{eq:les}
    \Theta_X(D) = \mathop{\arg\min}_{S \in \mathcal{S}_1} \| \phi(X)S - D \|^2.
  \end{equation}
\end{proposition}

\begin{proposition}\label{pro:rg}
  The Riemannian gradient of $ f $ at $ X \in \mathcal{M} $ can be expressed as
  \begin{equation}\label{eq:rg}
    \operatorname{grad}f(X) = \nabla f(X) - \phi(X)\Theta_X(\nabla f(X)).
  \end{equation}
  Moreover, the Riemannian Hessian of $ f $ at $ X \in \mathcal{M} $ admits the following expression,
  \begin{equation}
    \operatorname{hess}f(X)[Z] = \mathcal{P}_{T_X\mathcal{M}} \left( \nabla^2 f(X)[Z] - \phi(Z)\Theta_X(\nabla f(X))\right).
  \end{equation}
\end{proposition}
\begin{proof}
  By substituting $ \nabla f(X) $ in to \eqref{pro}, we get
  \begin{equation*}
    \operatorname{grad}f(X) = \nabla f(X) - \phi(X)\Theta_X(\nabla f(X)).
  \end{equation*}
  Then according to \cite{absil2008optimization}, one can obtain that the Riemannian Hessian of $ f $ at a point $ X $ in $ \mathcal{M} $ is the linear mapping $ \operatorname{hess}f(X) $ of $ T_X\mathcal{M} $ into itself defined by
  \begin{align*}
    \operatorname{hess}f(X)[Z]  = & \nabla_Z^{\operatorname{grad}f(X)} \\
    = & \mathcal{P}_{T_X\mathcal{M}}\left({\rm D}\operatorname{grad}f(X)[Z]\right) \\
    = & \mathcal{P}_{T_X\mathcal{M}}\left( {\rm D}(\nabla f(X) - \phi(X)\Theta_X(\nabla f(X)))[Z] \right)\\
    = & \mathcal{P}_{T_X\mathcal{M}}\left( \nabla^2f(X)[Z] - \phi(Z)\Theta_X(\nabla f(X)) - \phi(X){\rm D}\Theta_X(\nabla f(X))[Z]\right),
  \end{align*}
  where $ \nabla $ in the first equality is the Riemannian connection on $ \mathcal{M} $. Notice that
  \begin{equation*}
    {\rm D}\Theta_X(\nabla f(X))[Z] = \lim_{t \rightarrow 0}\frac{\Theta_{X+ tZ}(\nabla f(X+ tZ)) - \Theta_X(\nabla f(X))}{t}.
  \end{equation*}
  Denote $ S_1 = \Theta_{X+ tZ}(\nabla f(X+ tZ)) $ and $ S_2 = \Theta_X(\nabla f(X)) $, it follows that there exist $ T_1 , T_2 \in \mathcal{G} $ such that
  \begin{equation*}
    S_1 - S_2 = (T_1 - T_2)^{\top} + \psi( T_1 - T_2 ) \in \{ T^{\top} + \psi(T) : T \in \mathcal{G} \}.
  \end{equation*}
  There is the conclusion that $ \phi(X){\rm D}\Theta_X(\nabla f(X))[Z] \in N_X\mathcal{M} $. Therefore, the Riemannian Hessian $ \operatorname{hess}f(X) $ can be further simplified as
  \begin{equation*}
    \operatorname{hess}f(X)[Z] = \mathcal{P}_{T_X\mathcal{M}}\left( \nabla^2f(X)[Z] - \phi(Z)\Theta_X(\nabla f(X)) \right).
  \end{equation*}
\end{proof}


%

\subsection{Optimality conditions}
In this subsection, we present the optimality conditions for \ref{ocp} as follows.

\begin{definition}[\cite{absil2008optimization}]
  we call $ X \in \mathcal{M} $ is a first-order stationary point of \ref{ocp} if $ \operatorname{grad}f(X) = 0 $.
\end{definition}

Next, we present the definition of the second-order optimality condition of \ref{ocp}.

\begin{assumption}\label{ass2}
  $ f $ is twice differentiable in $ \mathcal{F} $, i.e., $ \nabla^2 f(X) $ exists at every $ X \in \mathcal{F} $.
\end{assumption}

\begin{definition}[\cite{absil2008optimization}]
  If Assumption \ref{ass2} holds, we call $ X \in \mathcal{M} $ is a second-order stationary point of \ref{ocp} if and only if $ X $ is a first-order stationary point of \ref{ocp}, and for any $  Z \in {\rm{T}}_X\mathcal{M} $, it holds that
  \begin{equation}\label{eq:2or}
    \left\langle  Z, \operatorname{hess}f(X)[Z] \right\rangle \geq 0.
  \end{equation}
\end{definition}

\section{Construction of an exact penalty function}\label{sec:3}
In this section, inspired by the Riemannian constraint dissolving framework \cite{xiao2024dissolving}, we introduce a corresponding constraint dissolving function for the \ref{ocp}. Following the approach proposed in \cite{xiao2024dissolving}, the constraint dissolving method addresses the minimization of smooth optimization problems with equality constraints by reformulating them as unconstrained problems via an exact penalty function
\begin{equation*}
  f(\mathcal{A}(x)) + \frac{\beta}{2}\|c(x)\|^2.
\end{equation*}
where $ \{ x\in \mathbb{R}^n :c(x)=0\} $ denotes the feasible region and $ \mathcal{A} $ represents the constraint dissolving operator. This operator satisfies the following assumptions:
\begin{assumption}[\cite{xiao2024dissolving}]
    \begin{itemize}
      \item[1.] For any $ x \in \{ x\in \mathbb{R}^n :c(x)=0\} $, it holds that $ \mathcal{A}(x) = x $.
      \item[2.] The Jacobian of $ c(\mathcal{A}(x)) $ equals zero for any $ x \in \{ x\in \mathbb{R}^n :c(x)=0\} $.
    \end{itemize}
\end{assumption}
The constrained dissolving operator is the key to constructing the constrained dissolving function, and this section will start with the construction of the constrained dissolving operator.

\subsection{constraint dissolving mapping}
In this subsection, we will construct the constraint dissolving operator as \eqref{eq:A}. In the following, we give a key theorem to show that the operator $ \mathcal{A} $ constructed in this paper is eligible.

\begin{theorem}\label{Th:cdm}
  $ \mathcal{A} $ is a constraint dissolving mapping.
\end{theorem}
\begin{proof}
  The locally Lipschitz smoothness of $ \mathcal{A} $ is guaranteed by the Lipschitz smoothness of the self-adjoint linear mapping $ \phi $. Then, for any $ X \in \mathcal{M} := \left\{ X : C(X) = 0 \right\} $, it follows that $ \mathcal{A}(X) = X $.
  Moreover, according to Lemma \ref{le3}, we have
  \begin{equation*}
    {\rm D}(C \circ \mathcal{A})(X)[Z] = \frac{9}{4}{\rm D}C(X)[Z] - 3{\rm D}C(X)[Z] + \frac{3}{4}{\rm D}C(X)[Z] = 0.
  \end{equation*}
  Therefore, we can conclude that $ \mathcal{A} $ is a constraint dissolving mapping.
\end{proof}

According to \cite{xiao2024dissolving}, the constraint dissolving operator $ \mathcal{A} $ exhibits several desirable properties, including the idempotence of its differential operator and its non-extensiveness. These properties play a crucial role in ensuring the equivalence and stability of the transformed unconstrained optimization problem.

\begin{lemma}\label{le:mi}
  For any given $ X \in \mathcal{M} $ and $ Z, T \in \mathcal{F} $, it holds that
  \begin{equation*}
    {\rm D}\mathcal{A}(X)\left[ {\rm D}\mathcal{A}(X) [Z] \right] = {\rm D}\mathcal{A}(X) [Z],
  \end{equation*}
  and
  \begin{equation*}
    {\rm D}\mathcal{A}(X)^*\left[ {\rm D}\mathcal{A}(X)^* [T] \right] = {\rm D}\mathcal{A}(X)^* [T],
  \end{equation*}
\end{lemma}
\begin{proof}
  For any given $ X \in \mathcal{M} $ and $ Z \in \mathcal{F} $, we have
  \begin{equation}\label{m1}
    \begin{aligned}
        {\rm D}\mathcal{A}(X)[Z] = & \frac{3}{2}Z - \frac{1}{2}\left( Z\left(\phi(X) \right)^{\top}X + X\left(\phi(Z) \right)^{\top}X + X\left(\phi(X) \right)^{\top}Z \right)\\
         = & Z - \frac{1}{2}\left( X\left(\phi(Z) \right)^{\top}X + X\left(\phi(X) \right)^{\top}Z \right).
    \end{aligned}
  \end{equation}
  Furthermore, under Assumption \ref{ass_1}, it is possible to deduce that
  \begin{equation}\label{m2}
    \begin{aligned}
        & {\rm D}\mathcal{A}(X)\left[X\left(\phi(Z) \right)^{\top}X + X\left(\phi(X) \right)^{\top}Z\right] \\
        = & X\left(\phi(Z) \right)^{\top}X + X\left(\phi(X) \right)^{\top}Z - \frac{1}{2} X\left(\phi(X\left(\phi(Z) \right)^{\top}X + X\left(\phi(X) \right)^{\top}Z) \right)^{\top}X \\
        & - \frac{1}{2} X\left(\phi(X) \right)^{\top}(X\left(\phi(Z) \right)^{\top}X + X\left(\phi(X) \right)^{\top}Z)\\
        = & X\left(\phi(Z) \right)^{\top}X + X\left(\phi(X) \right)^{\top}Z - \frac{1}{2}X\left((\phi(X))^{\top}Z(\phi(X))^{\top} + (\phi(Z))^{\top}X(\phi(X))^{\top} \right)X \\
        & -  \frac{1}{2} X\left(\phi(X) \right)^{\top}(X\left(\phi(Z) \right)^{\top}X + X\left(\phi(X) \right)^{\top}Z)\\
        = & X\left(\phi(Z) \right)^{\top}X + X\left(\phi(X) \right)^{\top}Z - X(\phi(X))^{\top}Z - X\left(\phi(Z) \right)^{\top}X \\
        = & 0.
    \end{aligned}
  \end{equation}
  Combining \eqref{m1} and \eqref{m2}, we complete the proof that
  \begin{equation*}
    \begin{aligned}
        {\rm D}\mathcal{A}(X)\left[ {\rm D}\mathcal{A}(X) [Z] \right] = & {\rm D}\mathcal{A}(X) [Z] - \frac{1}{2}{\rm D}\mathcal{A}(X)\left[X\left(\phi(Z) \right)^{\top}X + X\left(\phi(X) \right)^{\top}Z\right]
        =  {\rm D}\mathcal{A}(X) [Z].
    \end{aligned}
  \end{equation*}
  On the other hand, since $ {\rm D}\mathcal{A}(X)^* $ is a concomitant of ${\rm D}\mathcal{A}(X)$, it follows that
  \begin{align*}
    \left\langle {\rm D}\mathcal{A}(X)^* [T], Z \right\rangle = & \left\langle  T, {\rm D}\mathcal{A}(X)[Z] \right\rangle \\
    = & \left\langle  T, {\rm D}\mathcal{A}(X)\left[ {\rm D}\mathcal{A}(X) [Z] \right] \right\rangle \\
    = & \left\langle {\rm D}\mathcal{A}(X)^*\left[ {\rm D}\mathcal{A}(X)^* [T] \right], Z \right\rangle.
  \end{align*}
  The proof is completed.
\end{proof}


We now present two key results. Their proofs are following with Lemmas 1 and 3 in \cite{xiao2024dissolving} and are therefore omitted.

\begin{lemma}\label{le:cbo}
  For any $ X \in \mathcal{M} $ and $ Y \in \Omega_X $, it holds that
  \begin{equation*}
 \frac{3}{2\alpha_\mathcal{A}L_\phi} \| C(Y) \| \leq \| Y-X \| \leq \frac{2}{\delta_{C,X}} \| C(Y) \| .
  \end{equation*}
\end{lemma}

\begin{theorem}\label{them:noe}
  The operator $ \mathcal{A} $ is non-expansive. For any given $ X \in \mathcal{M} $ and $ Y \in \Omega_X $, it holds that
  \begin{equation*}
    \| C(\mathcal{A}(Y)) \| \leq \frac{8\alpha_{C,\mathcal{A}}L_{\phi}}{\delta_{C,X}^2}\| C(Y) \|^2.
  \end{equation*}
\end{theorem}
Theorem \ref{them:noe} demonstrates that the operator $\mathcal{A}$ effectively reduces constraint violations. Specifically, for any point $Y$ sufficiently close to the manifold $\mathcal{M}$, the application of $\mathcal{A}$ pulls $Y$ quadratically closer to $\mathcal{M}$.

\subsection{constraint dissolving function}
In this subsection, we consider the constraint dissolving function according to the operator $ \mathcal{A} $ as follows.
\begin{equation}\tag{GOCDF}\label{cdf}
  h(X) = f(\mathcal{A}(X)) + \frac{\beta}{2}\left\| X^{\top}\phi(X) - I_p \right\|^2.
\end{equation}
Subsequently, we deduce the explicit form of $ \nabla h(X) $ and $ \nabla^2 h(X) $.
\begin{proposition}\label{pro:nah}
  For any $ X \in \mathcal{F} $, the gradient of $ h $ can be expressed as
  \begin{equation}\label{eq:eh}
    \nabla h(X) = \nabla f(\mathcal{A}(X))\left(\frac{3}{2}I_p - \frac{1}{2}X^{\top}\phi(X)\right) - \frac{1}{2}\phi(X)\Phi\left(\nabla f(\mathcal{A}(X))^{\top}X\right) + \beta \phi(X)\Phi\left(X^{\top}\phi(X) - I_p\right),
  \end{equation}
  where $ \Phi(T) = T^{\top} + \psi(T), \forall T \in \mathcal{G} $.
\end{proposition}
\begin{proof}
  Denote $ g(X) = f(\mathcal{A}(X)) $. First we aim at proving the gradient of $ g $. According to the Taylor expansion of $ g $ , we have
  \begin{align*}
    & g(X + \Delta X) - g(X)\\
    = & \left\langle \nabla f(\mathcal{A}(X)), {\rm D}\mathcal{A}(X)[\Delta X] \right\rangle + \mathcal{O}(\| \Delta X \|^2) \\
    = & \left\langle {\rm D}\mathcal{A}(X)^*[ \nabla f(\mathcal{A}(X))], \Delta X \right\rangle + \mathcal{O}(\| \Delta X \|^2) \\
    = & \left\langle \nabla f(\mathcal{A}(X))\left(\frac{3}{2}I_p - \frac{1}{2}X^{\top}\phi(X)\right) - \frac{1}{2}\phi(X)\left(  X^{\top}\nabla f(\mathcal{A}(X)) + \psi\left(\nabla f(\mathcal{A}(X))^{\top}X\right)  \right), \Delta X \right\rangle + \mathcal{O}(\| \Delta X \|^2),
  \end{align*}
  which illustrates that
  \begin{equation}\label{eq:ng}
    \nabla g(X) = \nabla f(\mathcal{A}(X))\left(\frac{3}{2}I_p - \frac{1}{2}X^{\top}\phi(X)\right) - \frac{1}{2}\phi(X)\left(  X^{\top}\nabla f(\mathcal{A}(X)) + \psi\left(\nabla f(\mathcal{A}(X))^{\top}X\right)  \right).
  \end{equation}

  Moreover, since
  \begin{align*}
     & \left\| (X+\Delta X)^{\top}\phi(X + \Delta X) - I_p \right\|^2 - \left\| X^{\top}\phi(X) - I_p \right\|^2 \\
    = & 2\left\langle \Delta X, \phi(X)(\phi(X))^{\top}X + \phi(X)\psi(X^{\top}\phi(X)) - 2\phi(X)  \right\rangle +  \mathcal{O}(\| \Delta X \|^2),
  \end{align*}
  combined with the fact that $ h(X) = g(X) + \frac{\beta}{2}\left\| X^{\top}\phi(X) - I_p \right\|^2 $, we could conclude that
  \begin{align*}
    \nabla h(X) = & \nabla f(\mathcal{A}(X))\left(\frac{3}{2}I_p - \frac{1}{2}X^{\top}\phi(X)\right) - \frac{1}{2}\phi(X)\left(  X^{\top}\nabla f(\mathcal{A}(X)) + \psi\left(\nabla f(\mathcal{A}(X))^{\top}X\right)  \right) \\
    & + \beta \phi(X)\left((\phi(X))^{\top}X + \psi(X^{\top}\phi(X)) - 2I_p \right).
  \end{align*}
  Denote $ \Phi(T) = T^{\top} + \psi(T), \forall T \in \mathcal{G} $ and complete the proof.
\end{proof}

\begin{proposition}\label{pro:2h}
  If Assumption \ref{ass2} hold, then
  \begin{equation}\label{eq:nabh}
    \nabla^2 h(X)[\Delta X] = \nabla^2 g(X)[\Delta X ] + \beta \left( \phi(\Delta X)\Phi\left( X^{\top}\phi(X) - I_p \right) + \phi(X)\Phi\left(\Delta X^{\top}\phi(X)\right) + \phi(X)\Phi\left(X^{\top}\phi(\Delta X) \right) \right),
  \end{equation}
  where $ g(X) = f(\mathcal{A}(X)) $ and $ \Phi(T) = T^{\top} + \psi(T), \forall T \in \mathcal{G} $.
\end{proposition}
\begin{proof}
  According to the Taylor expansion of $ f $ , we have
  \begin{align*}
     & f\left(\mathcal{A}(X+\Delta X)\right) - f(\mathcal{A}(X)) \\
    = & f\left( \mathcal{A}(X) + {\rm D}\mathcal{A}(X)[\Delta X] - \frac{1}{2}\left( \Delta X\left(\phi(\Delta X) \right)^{\top}X + \Delta X\left(\phi(X) \right)^{\top}\Delta X + X\left(\phi(\Delta X) \right)^{\top}\Delta X \right) \right)\\
    & - f(\mathcal{A}(X)) + \mathcal{O}(\| \Delta X \|^3)\\
    = & \left\langle \nabla f(\mathcal{A}(X)), {\rm D}\mathcal{A}(X)[\Delta X] - \frac{1}{2}\left( \Delta X\left(\phi(\Delta X) \right)^{\top}X + \Delta X\left(\phi(X) \right)^{\top}\Delta X + X\left(\phi(\Delta X) \right)^{\top}\Delta X \right)  \right\rangle \\
    & + \frac{1}{2}\left\langle \nabla^2 f(\mathcal{A}(X))\left[{\rm D}\mathcal{A}(X)[\Delta X]\right], {\rm D}\mathcal{A}(X)[\Delta X] \right\rangle + \mathcal{O}(\| \Delta X \|^3)\\
    = & \left\langle \Delta X, \nabla g(X) \right\rangle - \frac{1}{2}\left\langle \nabla f(\mathcal{A}(X)), \Delta X\left(\phi(\Delta X) \right)^{\top}X + \Delta X\left(\phi(X) \right)^{\top}\Delta X + X\left(\phi(\Delta X) \right)^{\top}\Delta X \right\rangle \\
    & + \frac{1}{2}\left\langle \nabla^2 f(\mathcal{A}(X))\left[{\rm D}\mathcal{A}(X)[\Delta X]\right], {\rm D}\mathcal{A}(X)[\Delta X] \right\rangle + \mathcal{O}(\| \Delta X \|^3)\\
    = & \left\langle \Delta X, \nabla g(X) \right\rangle - \frac{1}{4}\left\langle \Delta X , \nabla f(\mathcal{A}(X))X^{\top}\phi(\Delta X) + \nabla f(\mathcal{A}(X))\Delta X^{\top}\phi(X) + \phi(\Delta X)X^{\top}\nabla f(\mathcal{A}(X)) \right\rangle \\
    & - \frac{1}{4}\left\langle \Delta X ,\phi\left( X \nabla f(\mathcal{A}(X))^{\top} \Delta X\right) + \phi(X)\Delta X^{\top} \nabla f(\mathcal{A}(X)) + \phi\left(\Delta X \nabla f(\mathcal{A}(X))^{\top} X\right) \right\rangle \\
    & + \frac{1}{2}\left\langle {\rm D}\mathcal{A}(X)^*\left[ \nabla^2 f(\mathcal{A}(X))\left[{\rm D}\mathcal{A}(X)[\Delta X]\right]\right], \Delta X \right\rangle + \mathcal{O}(\| \Delta X \|^3).
  \end{align*}
  Therefore, we conclude that
  \begin{equation*}
    \begin{aligned}
    \nabla^2 g(X)[\Delta X ] = & {\rm D}\mathcal{A}(X)^*\left[ \nabla^2 f(\mathcal{A}(X))\left[{\rm D}\mathcal{A}(X)[\Delta X]\right]\right] + {\rm D}^2 \mathcal{A}(X)^*\left[ \Delta X, \nabla f(\mathcal{A}(X)) \right]\\
     = & {\rm D}\mathcal{A}(X)^*\left[ \nabla^2 f(\mathcal{A}(X))\left[{\rm D}\mathcal{A}(X)[\Delta X]\right]\right] - \frac{1}{2}\nabla f(\mathcal{A}(X))\Phi((\phi(\Delta X))^{\top}X) \\
    & - \frac{1}{2}\phi(X)\Phi(\nabla f(\mathcal{A}(X))^{\top}\Delta X)  - \frac{1}{2}\phi(\Delta X)\Phi(\nabla f(\mathcal{A}(X))^{\top}X),
    \end{aligned}
  \end{equation*}
  where $ \Phi(T) = T^{\top} + \psi(T), \forall T \in \mathcal{G} $.

  Moreover, since
  \begin{align*}
     & \left\| (X+\Delta X)^{\top}\phi(X + \Delta X) - I_p \right\|^2 - \left\| X^{\top}\phi(X) - I_p \right\|^2 \\
    = & 2\left\langle \Delta X, \phi(X)\Phi\left(X^{\top}\phi(X) - I_p\right)  \right\rangle + 2\left\langle \Delta X^{\top}\phi(\Delta X),  X^{\top}\phi(X) - I_p \right\rangle + 2\left\langle \Delta X^{\top}\phi(X),X^{\top}\phi(\Delta X) \right\rangle \\
    & + \left\langle \Delta X^{\top}\phi(X),\Delta X^{\top}\phi(X) \right\rangle + \left\langle X^{\top}\phi(\Delta X),X^{\top}\phi(\Delta X) \right\rangle + \mathcal{O}(\| \Delta X \|^3),
  \end{align*}
  the hessian of $h(X)$ can be expressed as
  \begin{equation*}
    \nabla^2 h(X)[\Delta X] = \nabla^2 g(X)[\Delta X ] + \beta \left( \phi(\Delta X)\Phi\left( X^{\top}\phi(X) - I_p \right) + \phi(X)\Phi\left(\Delta X^{\top}\phi(X)\right) + \phi(X)\Phi\left(X^{\top}\phi(\Delta X) \right) \right).
  \end{equation*}
\end{proof}

\subsection{Basic properties of \ref{cdf}}
In this subsection, we discuss the relationship between \ref{ocp} and \ref{cdf}, particularly the correspondence between their first-order and second-order stationary points. In fact, since the mapping $\mathcal{A}$ has been proven to be an effective constraint dissolving operator in Theorem \ref{Th:cdm}, Proposition 4 and Theorems 1 and 2 in \cite{xiao2024dissolving} imply that, when the penalty parameter $ \beta $ is sufficiently large, \ref{ocp} and corresponding \ref{cdf} share the same first-order and second-order stationary points in a neighborhood of $ \mathcal{M} $. The main objective of this subsection is to derive a penalty parameter threshold corresponding to \ref{cdf}.

The relationship between the first-order stationary points of the \ref{ocp} and the \ref{cdf} on $ \mathcal{M} $ is presented below. The proof closely follows the argument of Proposition 4 in \cite{xiao2024dissolving} and is therefore omitted.

\begin{proposition}\label{pro:sta}
  For any $ X \in \mathcal{M} $, $ X $ is a first-order stationary point of \ref{ocp} if and only if $ X $ is a first-order stationary point of \ref{cdf}.
\end{proposition}
%

We next introduce the threshold $\tilde{\beta}_X$ for the penalty parameter $ \beta $ associated with the \ref{cdf} formulation, which is defined as follows:
\begin{equation*}
    \tilde{\beta}_X \geq \max  \left\{    \frac{2\alpha_{\mathcal{A}}^3L_\phi^2M_{X,f} + 30 \alpha_{\mathcal{A}}L_\phi M_{X,f} }{3\delta_{C,X}^2}, \frac{16\alpha_\mathcal{A}\alpha_{C,\mathcal{A}} L_\phi M_{X,f}}{3\delta_{C,X}^2},  \frac{18\alpha_\mathcal{A} L_\phi M_{X,f}+3\alpha_\mathcal{A}^2 L_\phi^2 M_{X,f} + 2 \alpha_\mathcal{A}^3 L_\phi^2 M_{X,f} +1}{9\delta_{C,X}^2 - 9\delta_{C,X}^2\alpha_{C,\mathcal{A}}} \right\}.
\end{equation*}



\begin{theorem}\label{the:first}
  Suppose Assumptions \ref{ass_1} and \ref{ass_2} hold, and $ \beta > \tilde{\beta}_X $. Then for any given $X \in \mathcal{M}$, any first-order stationary point of \ref{cdf} in $ \Omega_X $ is a first-order stationary point of \ref{ocp}.
\end{theorem}
\begin{proof}
  For any $ Y \in \Omega_X $, it holds that
  \begin{equation*}
    \left\| {\rm D}\mathcal{A}(Y)^*[\nabla h(Y)] - \nabla h(Y)  \right\| \leq  \left( \frac{1}{6}\alpha_{\mathcal{A}}^2L_\phi + \frac{5}{2} \right) \| \nabla h(Y) \|.
  \end{equation*}
  Considering $ \nabla h(Y) = {\rm D}\mathcal{A}(Y)^*[\nabla f (\mathcal{A}(Y))] + \beta {\rm D}C(Y)^*[C(Y)] $, we simplify both terms in turn. For the first, we have
    \begin{equation*}
    \begin{aligned}
         & \left\| {\rm D}\mathcal{A}(Y)^*\left[{\rm D}\mathcal{A}(Y)^*[\nabla f (\mathcal{A}(Y))]\right] - {\rm D}\mathcal{A}(Y)^*[\nabla f (\mathcal{A}(Y))]  \right\| \\
        = & {\Big\|}  {\rm D}\mathcal{A}(Y)^*\left[{\rm D}\mathcal{A}(Y)^*[\nabla f (\mathcal{A}(Y))]\right] - {\rm D}\mathcal{A}(X)^*\left[{\rm D}\mathcal{A}(X)^*[\nabla f (\mathcal{A}(Y))]\right]  \\
        & +  {\rm D}\mathcal{A}(X)^*[\nabla f (\mathcal{A}(Y))] - {\rm D}\mathcal{A}(Y)^*[\nabla f (\mathcal{A}(Y))]  {\Big\|} \\
        \leq & \left(  \alpha_{\mathcal{A}}L_\phi \left\| {\rm D}\mathcal{A}(Y)^*[\nabla f (\mathcal{A}(Y))]\right\| + \alpha_{\mathcal{A}}L_\phi \left\| \nabla f (\mathcal{A}(Y))\right\|  \right)\| Y - X \| \\
        \leq &  \alpha_{\mathcal{A}}L_\phi \left( \frac{5}{2} + \frac{1}{6}\alpha_{\mathcal{A}}^2L_\phi \right) \| \nabla f (\mathcal{A}(Y))\| \| Y - X \| \\
        \leq & \frac{\alpha_{\mathcal{A}}^3L_\phi^2M_{X,f} + 15 \alpha_{\mathcal{A}}L_\phi M_{X,f} }{3\delta_{C,X}}\|C(Y)\|,
    \end{aligned}
  \end{equation*}
  where the first equation follows from Lemma \ref{le:mi} and the last inequality follows from Lemma \ref{le:cbo}. In the same way, we simplify the second term yields
  \begin{align*}
     & \left\| {\rm D}\mathcal{A}(Y)^*\left[{\rm D}C(Y)^*[C(Y)]\right] - {\rm D}C(Y)^*[C(Y)]  \right\| \\
    \geq & \left\| {\rm D}C(Y)^*[C(Y)]  \right\| - \left\| {\rm D}\mathcal{A}(Y)^*\left[{\rm D}C(Y)^*[C(Y)]\right]  \right\| \\
    \geq & \delta_{C,X}\| C(Y) \| - \left\| {\rm D}\mathcal{A}(Y)^*\left[{\rm D}C(Y)^*[C(Y)]\right] - {\rm D}\mathcal{A}(X)^*\left[{\rm D}C(\mathcal{A}(X))^*[C(Y)]\right]  \right\| \\
    \geq & \delta_{C,X}\| C(Y) \| - \left(\frac{2}{3}\alpha_{\mathcal{A}}^2L_\phi^2 + 3L_\phi + \frac{1}{3}\alpha_{\mathcal{A}}^2L_\phi^2 \right)\| Y - X \| \| C(Y) \| \\
    \geq & \left(\delta_{C,X} - \left(\alpha_{\mathcal{A}}^2L_\phi^2 + 3L_\phi\right) \varepsilon_X \right)\| C(Y) \| \\
    \geq & \frac{\delta_{C,X}}{2}\| C(Y) \|.
  \end{align*}
  Combining the two inequalities, we have that
  \begin{align*}
    \| \nabla h(Y) \| \geq & \frac{6}{\alpha_{\mathcal{A}}^2L_\phi + 15} \left\| {\rm D}\mathcal{A}(Y)^*[\nabla h(Y)] - \nabla h(Y)  \right\| \\
    \geq & \frac{6 \beta}{\alpha_{\mathcal{A}}^2L_\phi + 15} \left\| {\rm D}\mathcal{A}(Y)^*\left[{\rm D}C(Y)^*[C(Y)]\right] - {\rm D}C(Y)^*[C(Y)]  \right\| \\
     & - \frac{6}{\alpha_{\mathcal{A}}^2L_\phi + 15} \left\| {\rm D}\mathcal{A}(Y)^*\left[{\rm D}\mathcal{A}(Y)^*[\nabla f (\mathcal{A}(Y))]\right] - {\rm D}\mathcal{A}(Y)^*[\nabla f (\mathcal{A}(Y))]  \right\|\\
    > & \frac{6}{\alpha_{\mathcal{A}}^2L_\phi + 15} \left(\frac{\beta\delta_{C,X}}{2} - \frac{\alpha_{\mathcal{A}}^3L_\phi^2M_{X,f} + 15 \alpha_{\mathcal{A}}L_\phi M_{X,f} }{3\delta_{C,X}} \right)\| C(Y) \|.
  \end{align*}
  Assuming that $ X^* \in \Omega_X $ is a first-order stationary point of the \ref{cdf}, we have $ \nabla h(X^*) = 0 $. Therefore, when $ \beta > \tilde{\beta}_X $, the first-order stationary point $ X^* $ must satisfy $ \|C(X^*)\| = 0 $, implying that $ X^* \in \mathcal{M} $. By Proposition \ref{pro:sta}, it follows that $ X^* $ is also a first-order stationary point of \ref{ocp}.
\end{proof}

In the following we discuss the relationship between the second-order stationary points of the \ref{cdf} and the \ref{ocp}.
\begin{lemma}\label{le:sec1}
  Suppose Assumption \ref{ass2} holds. Then for any $ X \in \mathcal{M} $ and $ Z \in T_X\mathcal{M} $, it holds that
  \begin{equation*}
    {\rm D}^2 \mathcal{A}(X)^*\left[ Z, \operatorname{grad}f(X) \right] = {\rm D}^2 \mathcal{A}(X)^*\left[ Z, \nabla f(X) \right] + \frac{1}{2}\phi(Z)\Phi\left( \Theta_X(\nabla f(X))^{\top} \right).
  \end{equation*}
  Moreover, if $ X $ is a first-order stationary point of \ref{ocp}, then for any $ Z \in T_X\mathcal{M} $, it holds that
  \begin{equation*}
    {\rm D}^2 \mathcal{A}(X)^*\left[ Z, \nabla f(X) \right] = - \frac{1}{2}\phi(Z)\Phi\left( \Theta_X(\nabla f(X))^{\top}\right).
  \end{equation*}
\end{lemma}
\begin{proof}
  For any $ X \in \mathcal{M} $ and $ Z \in T_X\mathcal{M} $, it follows from Propositions \ref{pro:rg} and \ref{pro:2h} that
  \begin{align*}
    {\rm D}^2 \mathcal{A}(X)^*\left[ Z, \operatorname{grad}f(X) \right] = & {\rm D}^2 \mathcal{A}(X)^*\left[ Z, \nabla f(X) - \phi(X)\Theta_X(\nabla f(X)) \right] \\
    = & {\rm D}^2 \mathcal{A}(X)^*\left[ Z, \nabla f(X) \right] - {\rm D}^2 \mathcal{A}(X)^*\left[ Z, \phi(X)\Theta_X(\nabla f(X)) \right]\\
    = & {\rm D}^2 \mathcal{A}(X)^*\left[ Z, \nabla f(X) \right] + \frac{1}{2}\phi(X)\Theta_X(\nabla f(X))\Phi\left((\phi(Z))^{\top}X\right) \\
    & + \frac{1}{2}\phi(X)\Phi\left(\Theta_X(\nabla f(X))^{\top}\phi(X)^{\top}Z\right)  + \frac{1}{2}\phi(Z)\Phi \left(\Theta_X(\nabla f(X))^{\top}\phi(X)^{\top}X\right) \\
    = & {\rm D}^2 \mathcal{A}(X)^*\left[ Z, \nabla f(X) \right] + \frac{1}{2}\phi(Z)\Phi\left( \Theta_X(\nabla f(X))^{\top} \right).
  \end{align*}
  Here the last equation holds because of the fact that $ \phi(X)\Theta_X(\nabla f(X)) \in N_X\mathcal{M} $ and $ \Phi\left(\phi(Z)^{\top}X\right) = 0 $. In particular, when $ X $ is a first-order stationary point of the \ref{ocp}, we have $ \operatorname{grad}f(X) = 0 $, which implies $ {\rm D}^2 \mathcal{A}(X)^*\left[ Z, \operatorname{grad}f(X) \right] = 0 $. We complete the proof.
\end{proof}

\begin{lemma}\label{le:sec2}
  Suppose Assumption \ref{ass2} holds. Then for any first-order stationary point $ X \in \mathcal{M} $ of \ref{ocp}, $ Z \in T_X\mathcal{M} $ and $ S \in N_X\mathcal{M} $, it holds that
  \begin{equation*}
    \left\langle S, \nabla^2 h (X)[Z] \right\rangle = 0
  \end{equation*}
\end{lemma}
\begin{proof}
  The proof starts at $ \phi(Z) \in T_X\mathcal{M} $. Since $ Z \in T_X\mathcal{M} $, we have
  \begin{equation*}
     X^{\top} \phi(Z)  +  Z^{\top} \phi(X) = 0,
  \end{equation*}
  which means that
  \begin{equation*}
    \psi\left( \phi(X)^{\top}Z + \psi(\phi(Z)^{\top}X) \right) = 0,
  \end{equation*}
  i.e., $ \phi(Z) \in T_X\mathcal{M} $. If $X$ is a first-order stationary point of the \ref{ocp}, then by Proposition \ref{pro:2h} and Lemma \ref{le:sec1} we have
  \begin{align*}
    \left\langle S, \nabla^2 h (X)[Z] \right\rangle = & \left\langle S, {\rm D}\mathcal{A}(X)^*\left[ \nabla^2 f(\mathcal{A}(X))\left[{\rm D}\mathcal{A}(X)[Z]\right]\right] + {\rm D}^2 \mathcal{A}(X)^*\left[ Z, \nabla f(\mathcal{A}(X)) \right] \right\rangle \\
    & + \beta \left\langle S, \left( \phi(Z)\Phi\left( X^{\top}\phi(X) - I_p \right) + \phi(X)\Phi\left(Z^{\top}\phi(X)\right) + \phi(X)\Phi\left(X^{\top}\phi(Z) \right) \right) \right\rangle\\
    = & \left\langle S, {\rm D}\mathcal{A}(X)^*\left[ \nabla^2 f(X)\left[Z\right]\right] \right\rangle + \left\langle S, {\rm D}^2 \mathcal{A}(X)^*\left[ Z, \nabla f(X) \right] \right\rangle \\
    = & - \frac{1}{2}\left\langle s, \phi(Z)\Phi\left( \Theta_X(\nabla f(X))^{\top} \right) \right\rangle \\
    = & 0.
  \end{align*}
  The proof is completed.
\end{proof}

\begin{theorem}
  Suppose Assumption \ref{ass2} holds. For any given $ X \in \mathcal{M} $, if $ \beta > \tilde{\beta}_X $ then \ref{cdf} and \ref{ocp} share the same second-order stationary points over $ \Omega_X $.
\end{theorem}
\begin{proof}
  First, if $Y \in \Omega_X$ is a second-order stationary point of \ref{cdf}, for any $ Z \in T_Y \mathcal{M} $, we have $ \left\langle  Z, \nabla^2 h (Y)[Z] \right\rangle \geq 0 $. Then it follows from theorem \ref{the:first} that $ Y \in \mathcal{M} $ and that it is a first-order stationary point of the \ref{ocp}. Combining Proposition \ref{pro:2h} and Lemma \ref{le:sec1} we obtain that
  \begin{align*}
    \left\langle Z, \nabla^2 h (Y)[Z] \right\rangle = & \left\langle Z, {\rm D}\mathcal{A}(Y)^*\left[ \nabla^2 f(\mathcal{A}(Y))\left[{\rm D}\mathcal{A}(Y)[Z]\right]\right] + {\rm D}^2 \mathcal{A}(Y)^*\left[ Z, \nabla f(\mathcal{A}(Y)) \right] \right\rangle \\
    & + \beta \left\langle Z, \left( \phi(Z)\Phi\left( Y^{\top}\phi(Y) - I_p \right) + \phi(Y)\Phi\left(Z^{\top}\phi(Y)\right) + \phi(Y)\Phi\left(Y^{\top}\phi(Z) \right) \right) \right\rangle\\
    = & \left\langle Z, {\rm D}\mathcal{A}(Y)^*\left[ \nabla^2 f(Y)\left[Z\right]\right] \right\rangle + \left\langle Z, {\rm D}^2 \mathcal{A}(Y)^*\left[ Z, \nabla f(Y) \right] \right\rangle \\
    = & \left\langle Z, \nabla^2 f(Y)\left[Z\right] \right\rangle - \frac{1}{2}\left\langle Z, \phi(Z)\Phi\left( \Theta_Y(\nabla f(Y))^{\top} \right) \right\rangle \\
    = & \left\langle Z, \nabla^2 f(Y)\left[Z\right] \right\rangle - \frac{1}{2}\left\langle Z, \phi(Z)\Phi\left( \nabla f(Y)^{\top} Y \right) \right\rangle.
  \end{align*}
  Further, for any $ Z \in T_Y \mathcal{M} $, we have
  \begin{align*}
    \left\langle Z, \operatorname{hess}f(Y)[Z] \right\rangle = & \left\langle Z, \mathcal{P}_{T_Y\mathcal{M}} \left( \nabla^2 f(Y)[Z] - \phi(Z)\Theta_Y(\nabla f(Y))\right) \right\rangle \\
    = & \left\langle Z,  \nabla^2 f(Y)[Z] - \phi(Z)\Theta_Y(\nabla f(Y)) \right\rangle\\
    = & \left\langle Z,  \nabla^2 f(Y)[Z]  \right\rangle - \left\langle Z,  \phi(Z)Y^{\top}\nabla f(Y) \right\rangle,
  \end{align*}
  where the last equation is due to $ \operatorname{grad}f(Y) = 0 $, i.e., $ \Theta_Y(\nabla f(Y)) = Y^{\top}\nabla f(Y) $. Since $ \phi $ is a self-adjoint linear operator, it follows that
  \begin{equation*}
    \left\langle Z,  \phi(Z)Y^{\top}\nabla f(Y) \right\rangle = \left\langle Z\nabla f(Y)^{\top}Y,  \phi(Z) \right\rangle = \left\langle \phi(Z)\psi(\nabla f(Y)^{\top}Y),  Z \right\rangle.
  \end{equation*}
  Thus we obtain the conclusion
  \begin{equation*}
    \left\langle Z, \operatorname{hess}f(Y)[Z] \right\rangle = \left\langle  Z, \nabla^2 h (Y)[Z] \right\rangle \geq 0.
  \end{equation*}
  which implies that $ Y $ is a second-order stationary point of \ref{ocp}.

  On the other hand, since \ref{cdf} is defined on $ \mathcal{F} $, to verify that $Y$ is a second-order stationary point of \ref{cdf}, it is necessary to show that for any $ T \in \mathcal{F} $, the inequality $ \left\langle T, \nabla^2 h (Y)[T] \right\rangle \geq 0 $ holds. Noting that $ \mathcal{M} $ is an embedded submanifold of $ \mathcal{F} $, any $ T \in \mathcal{F}$ can be decomposed as $ T = Z + S $, where $ Z \in T_Y \mathcal{M} $ and $ S \in N_Y \mathcal{M} $. Moreover, if $Y$ is a second-order stationary point of \ref{ocp}, for any $ Z \in T_Y \mathcal{M} $, it holds that $ \left\langle Z, \operatorname{hess}f(Y)[Z] \right\rangle \geq 0 $. Following the argument in the previous proof, we similarly obtain $ \left\langle Z, \nabla^2 h (Y)[Z] \right\rangle \geq 0 $.

  Next, for any $ S \in N_Y \mathcal{M} $, we have
  \begin{align*}
    & \left\langle S, \nabla^2 h (Y)[S] \right\rangle \\
     = & \left\langle S, {\rm D}\mathcal{A}(Y)^*\left[ \nabla^2 f(\mathcal{A}(Y))\left[{\rm D}\mathcal{A}(Y)[S]\right]\right] + {\rm D}^2 \mathcal{A}(Y)^*\left[ S, \nabla f(\mathcal{A}(Y)) \right] \right\rangle \\
    & + \beta \left\langle S, \left( \phi(S)\Phi\left( Y^{\top}\phi(Y) - I_p \right) + \phi(Y)\Phi\left(S^{\top}\phi(Y)\right) + \phi(Y)\Phi\left(Y^{\top}\phi(S) \right) \right) \right\rangle \\
    = & \beta \left\langle S, {\rm D}C(Y)^*[{\rm D}C(Y)[S]] \right\rangle - \frac{1}{2}\left\langle S, \nabla f(Y)\Phi(\phi(S)^{\top}Y) + \phi(Y)\Phi(\nabla f(Y)^{\top}S) + \phi(S) \Phi(\nabla f(Y)^{\top}Y) \right\rangle\\
    \geq & \beta \delta_{C,X}^2\|S\|^2 - 3 L_\phi \| Y \|\| \nabla f(Y) \|\|S\|^2 \\
    > & \left( \beta \delta_{C,X}^2 - \alpha_\mathcal{A} L_\phi M_{X,f} \right)\|S\|^2   > 0.
  \end{align*}
  where the penultimate inequality holds due to $ \mathcal{A}(Y) = Y $ and the last inequality holds due to $ \beta > \frac{2\alpha_{\mathcal{A}}^3L_\phi^2M_{X,f} + 30 \alpha_{\mathcal{A}}L_\phi M_{X,f} }{3\delta_{C,X}^2} > \frac{\alpha_{\mathcal{A}}L_\phi M_{X,f}}{\delta_{C,X}^2} $.
  When $ \beta > \tilde{\beta}_X $, it follows from Lemma \ref{le:sec2} that
  \begin{equation*}
    \left\langle T, \nabla^2 h (Y)[T] \right\rangle = \left\langle Z, \nabla^2 h (Y)[Z] \right\rangle + 2\left\langle S, \nabla^2 h (Y)[Z] \right\rangle + \left\langle S, \nabla^2 h (Y)[S] \right\rangle > 0.
  \end{equation*}
  Therefore, $ Y $ is a second-order stationary point of \ref{cdf}.
\end{proof}

\section{Improved penalty function approaches for solving \ref{ocp}}\label{sec:4}
This section revisits the algorithmic framework for solving optimization problems with Riemannian manifold constraints as described in \cite{xiao2024dissolving}, which also serves as the foundation for addressing the OCP problem. The details are given in Algorithm \ref{al:1}. In the subsections that follow, we analyze the stability properties of infeasible points within this framework and examine its computational complexity.

\begin{algorithm}[H]\label{al:1}
\caption{Improved Penalty Function Approaches for \ref{ocp}.}
\KwIn{Linear mapping $ \phi $, objective function $f$, penalty parameter $\beta$, initial point $ X_0 $, stationarity tolerance $ \varepsilon_s $, feasibility tolerance $ \varepsilon_f $.}
\KwOut{The solution $ {X^*} \in \mathcal{M} $ of \ref{ocp}.}
Construct the penalty function \ref{cdf};\\
Choose an existing unconstrained method to generate the sequence $ \{X_k\} $ starting at the initial point $ X_0 $. Stop when the tolerance $ \varepsilon_s $ is reached to obtain $ \tilde{X} $;\\
\While{$\|C(\tilde{X})\| \geq \varepsilon_f $}{$\tilde{X} = \mathcal{A}(\tilde{X})$;}
\Return $ X^* = \tilde{X} $.
\end{algorithm}

\begin{remark}
  Algorithm \ref{al:1} is an infeasible method; therefore, directly applying an unconstrained optimization method to solve the CDF (i.e., Step 2 of Algorithm 1) may yield a solution that does not meet the desired feasibility requirements. In such cases, it becomes necessary to execute Steps 3–5 as a post-processing procedure. According to Theorem \ref{them:noe}, this additional processing can further reduce the degree of constraint violation and improve feasibility.
\end{remark}

\subsection{Stationarity at infeasible points}
In this subsection, we investigate the stationarity of \ref{cdf}. Specifically, we aim to establish a connection between the optimality and feasibility of the solution of \ref{ocp} through $ \|\nabla h(Y)\| $. Furthermore, we analyze the impact of Steps 3–5 in Algorithm \ref{al:1} on improving the quality of the solution, particularly in terms of enhancing its feasibility and convergence accuracy.

\begin{proposition}\label{pro:con}
  For any given $ X \in \mathcal{M} $, suppose $ \beta \geq \tilde{\beta}_X $ and $ Y \in \Omega_X $, then, it holds that
  \begin{equation*}
    \|\nabla h(Y)\|^2  \geq \| \nabla g(Y)\|^2 + \frac{\beta}{9}\| C(Y) \|.
  \end{equation*}
\end{proposition}
\begin{proof}
  For any $ Y \in \Omega_X $, by definition of $ \nabla g(Y) $ in \eqref{eq:ng}, we have
  \begin{align*}
     & \|\nabla g(Y)\| \| {\rm D}C(Y)^*[C(Y)] \| \\
    = & \left(\left\| \nabla f(\mathcal{A}(X))\left(\frac{3}{2}I_p  -
    \frac{1}{2}Y^\top \phi(Y) \right)  - \frac{1}{2} {\rm D}C(Y)^*[\nabla f(\mathcal{A}(X))^\top Y] \right\|  \right) \| {\rm D}C(Y)^*[C(Y)] \|\\
    \leq & \left(\left\| \nabla f(\mathcal{A}(X))\left(\frac{3}{2}I_p  -
    \frac{1}{2}Y^\top \phi(Y) \right) \right\|  + \frac{1}{2}\| {\rm D}C(Y)^*[\nabla f(\mathcal{A}(X))^\top Y] \|  \right) \| {\rm D}C(Y)^*[C(Y)] \|\\
    < & \left( 1 + \frac{1}{6}\alpha_\mathcal{A} L_\phi + \frac{1}{9}\alpha_\mathcal{A}^2 L_\phi \right) \alpha_\mathcal{A} L_\phi M_{X,f} \| C(Y) \|.
  \end{align*}
  It follows from Proposition \ref{pro:nah} that
  \begin{align*}
    \|\nabla h(Y)\|^2 = & \| \nabla g(Y)\|^2 + 2\beta \left\langle \nabla g(Y), {\rm D}C(Y)^*[C(Y)] \right\rangle + \beta^2\| {\rm D}C(Y)^*[C(Y)] \|^2 \\
    \geq & \| \nabla g(Y)\|^2 + \beta^2\| {\rm D}C(Y)^*[C(Y)] \|^2 - 2\beta \|\nabla g(Y)\| \| {\rm D}C(Y)^*[C(Y)] \| \\
    \geq & \| \nabla g(Y)\|^2 + \beta^2 \delta_{C,X}^2 \| C(Y) \|^2 - 2\beta \|\nabla g(Y)\| \| {\rm D}C(Y)^*[C(Y)] \| \\
    \geq & \| \nabla g(Y)\|^2 + \beta^2 \delta_{C,X}^2(1 - \|Y^\top \phi(Y)\|)\| C(Y) \| \\
    & - 2\beta \left( 1 + \frac{1}{6}\alpha_\mathcal{A} L_\phi + \frac{1}{9}\alpha_\mathcal{A}^2 L_\phi \right) \alpha_\mathcal{A} L_\phi M_{X,f} \| C(Y) \| \\
    > & \| \nabla g(Y)\|^2 + \frac{\beta}{9}\| C(Y) \|.
  \end{align*}
  Here, the last inequality uses the fact that $ \beta \geq \tilde{\beta}_X $.
\end{proof}

Proposition \ref{pro:con} shows that as the penalty parameter $ \beta $ increases, the solution to \ref{cdf} obtained via an unconstrained optimization method exhibits increasingly higher feasibility accuracy. In the following, we analyze how the post-processing steps (Steps 3–5) of Algorithm \ref{al:1} influence the value of the CDF objective function.

\begin{proposition}
  For any given $ X \in \mathcal{M} $, suppose $ \beta \geq \tilde{\beta}_X $ and $ Y \in \Omega_X $, then, it holds that
  \begin{equation*}
    h(\mathcal{A}(Y)) \leq h(Y) - \frac{\beta}{8}\| C(Y) \|^2.
  \end{equation*}
\end{proposition}
\begin{proof}
  For any $ Y \in \Omega_X $, according to the definition of $ \mathcal{A}(Y) $ in \eqref{eq:A}, we have
  \begin{equation*}
    \|\mathcal{A}(Y) - Y \| =  \| \frac{3}{2}Y - \frac{1}{2}Y \phi(Y)^{\top} Y - Y \|   = \frac{1}{2}\| Y C(Y)^{\top} \|
    \leq \frac{1}{6}\alpha_\mathcal{A}\| C(Y) \|.
  \end{equation*}
  Combining with the continuity of the function $f$, we deduce that
  \begin{align*}
    \left| f(\mathcal{A}^2(Y)) -  f(\mathcal{A}(Y)) \right| \leq &  M_{X,f}\| \mathcal{A}^2(Y) - \mathcal{A}(Y) \| \\
    \leq & \frac{\alpha_\mathcal{A}M_{X,f}}{6}\|C(\mathcal{A}(Y)) \| \\
    \leq & \frac{4\alpha_\mathcal{A}\alpha_{C,\mathcal{A}}L_{\phi}M_{X,f}}{3\delta_{C,X}^2}\| C(Y) \|^2.
  \end{align*}
  From Theorem \ref{them:noe} we have $ \| C(\mathcal{A}(Y)) \| \leq \frac{8\alpha_{C,\mathcal{A}}L_{\phi}}{\delta_{C,X}^2}\| C(Y) \|^2 \leq \frac{1}{2}\| C(Y) \| $. Therefore,
  \begin{align*}
    h(\mathcal{A}(Y)) - h(Y) \leq & \left| f(\mathcal{A}^2(Y)) -  f(\mathcal{A}(Y)) \right| + \frac{\beta}{2}\left( \| C(\mathcal{A}(Y)) \|^2 - \| C(Y) \|^2 \right) \\
    \leq & - \left( \frac{3\beta}{8} - \frac{4\alpha_\mathcal{A}\alpha_{C,\mathcal{A}}L_{\phi}M_{X,f}}{3\delta_{C,X}^2} \right)\| C(Y) \|^2 \\
    \leq & -\frac{\beta}{8}\| C(Y) \|^2.
  \end{align*}
  The proof is completed.
\end{proof}
\subsection{Computational complexity analysis}
In this subsection, we will analyze the computational complexity of Algorithm \ref{al:1} when it is embedded with a first-order optimization method such as gradient descent or conjugate gradient. In fact, the computational cost primarily arises from computing the gradient of the objective function $ h(X) $. Let $ O_f $ denote the cost of computing $ \nabla f(X) $, and $ O_\phi $ the cost of computing $ \phi(X) $. The computational cost of the basic linear algebra operations and the total cost of computing the gradient of $ h(X) $ are summarized in Table \ref{Ta:cost}.

\begin{table}
\centering
  \caption{Computational complexity the first-order oracle in Algorithm \ref{al:1}}
  \resizebox{0.6\textwidth}{!}
  {\begin{tabular}{c|c|c}
  	\toprule
\multirow{5}{*}{Compute $\nabla f(\mathcal{A}(X))\left(\frac{3}{2}I_p - \frac{1}{2}X^{\top}\phi(X)\right) $} & $\phi(X)$ & $ O_\phi $ \\
&  $X^{\top}\phi(X)$ & $ 2np^2 $ \\
&  $ \mathcal{A}(X)=\frac{3}{2}X - \frac{1}{2}X\phi(X)^{\top}X $ &  $ 2np^2 $  \\
&  $ \nabla f(\mathcal{A}(X)) $ & $ O_f $    \\
& $ \nabla f(\mathcal{A}(X))\left(\frac{3}{2}I_p - \frac{1}{2}X^{\top}\phi(X)\right) $ & $ 2np^2 $   \\ \midrule
\multirow{4}{*}{Compute $\phi(X)\Phi\left(\nabla f(\mathcal{A}(X))^{\top}X\right)$ }&  $\nabla f(\mathcal{A}(X))^{\top}X$ & $ 2np^2 $ \\
&  $ \phi(X)X^{\top}\nabla f(\mathcal{A}(X)) $ &  $ 2np^2 $  \\
&  $ X \nabla f(\mathcal{A}(X))^{\top}X  $ & $ 2np^2 $    \\
&  $ \phi\left( X \nabla f(\mathcal{A}(X))^{\top}X \right) $ & $ O_\phi $    \\ \midrule
\multirow{3}{*}{Compute $ \phi(X)\Phi\left(X^{\top}\phi(X) - I_p\right)$ } &  $\phi(X) \left( \phi(X)^\top X - I_p \right) $ & $ 2np^2 $ \\
&  $ X \left(X^{\top}\phi(X) - I_p\right)  $ & $ 2np^2 $    \\
&  $ \phi\left( X \left(X^{\top}\phi(X) - I_p\right)\right) $ & $ O_\phi $    \\ \midrule
In total& \multicolumn{2}{c}{$ O_f + 3O_\phi + 16np^2 $} \\
 \bottomrule
  \end{tabular}}
\label{Ta:cost}
\end{table}

Next, we compare the computational complexity of Algorithm \ref{al:1} embedded in a first-order optimization method with that of a standard first-order Riemannian optimization method. According to \eqref{eq:rg}, computing the Riemannian gradient involves solving a least squares problem \eqref{eq:les}. The computational cost of this process can vary significantly depending on $ \phi $. For instance, on the Stiefel manifold, the least squares problem can be solved explicitly via matrix symmetrization, resulting in a computational cost of $ 2np^2 $. However, on the indefinite Stiefel manifold, solving a Lyapunov equation with a $ p \times p $ coefficient matrix is required, leading to a computational complexity on the order of $ \mathcal{O}(n^2p+p^3) $. For a general mapping $ \phi $, the least squares problem involved in computing the Riemannian gradient becomes equivalent to solving a standard linear least squares system
\begin{equation*}
  \min_{{\rm vec}{(T)}}\| M{\rm vec}{(T)} - {\rm vec}{(D)} \|,
\end{equation*}
where $ M \in \mathbb{R}^{np \times p^2} $. The computational complexity in this case can grow to $ \mathcal{O}(np^4+p^6) $, especially when the special structure of the manifold $ \mathcal{M} $ is not exploited. This computational overhead occurs not only during the computation of the Riemannian gradient but also in the vector transport step. In contrast, Algorithm \ref{al:1} effectively avoids these costly operations. A summary of the computational complexity comparison is provided in Table \ref{Ta:comp}, using the indefinite Stiefel manifold as a representative case.

\begin{table}[H]
\centering
  \caption{Comparison of the computational complexity of first-order methods between Riemannian optimization approaches and Algorithm \ref{al:1} (for the case of the Indefinite Stiefel manifold $ {\rm iSt}(p,n) $).}
  \resizebox{0.8\textwidth}{!}
  {\begin{tabular}{cc|cc}
  	\toprule
 \multicolumn{2}{c|}{\small Riemannian optimization approaches} & \multicolumn{2}{c}{\small Algorithm \ref{al:1} with embedded first-order optimization approaches} \\ \midrule
 Riemannian gradient \eqref{eq:rg} & $ O_f + \mathcal{O}(n^2p+p^3) + 8np^2 $ & Euclidean gradient \eqref{eq:eh} & $ O_f + 28np^2 $ \\
 Retraction & $ \mathcal{O}(n^3)$ & No retraction & -- \\
 Vector transport & $ \mathcal{O}(n^2p+p^3)$ & No vector transport & -- \\
 \bottomrule
  \end{tabular}}
\label{Ta:comp}
\end{table}

\section{Numerical experiments}\label{sec:5}
In this section, we present the numerical performance of solving the \ref{ocp} problem using the \ref{cdf} framework, in different unconstrained solvers. All the numerical experiments in this section are run in serial on a workstation with two Intel(R) Xeon(R) Gold 5317 CPU @ 3.00 GHz under Ubuntu 20.04.1. running Python 3.8.0, NumPy 1.26.4 and the CDOpt \cite{xiao2025cdopt} packages.

\subsection{Basic settings}
We select three representative problems for testing:
\begin{itemize}
  \item The least squares matching problem on the symplectic Stiefel manifold.
  \item The extrinsic mean problem on the indefinite Stiefel manifold.
  \item The tensor joint $f$-diagonalization problem on the third-order tensor Stiefel manifold.
\end{itemize}
The construction details for each test instance are provided in the corresponding subsections below.

In the experiments, we used various unconstrained optimization solvers from the SciPy package \cite{virtanen2020scipy} to minimize the \ref{cdf} function corresponding to each problem. These include the conjugate gradient method \cite{nocedal1999numerical}, the limit-memory BFGS method \cite{byrd1995limited}, and the Trust-NCG method \cite{nocedal1999numerical}. In addition, we applied the gradient descent method with alternating Barzilai–Borwein (BB) steps \cite{fletcher2005barzilai}. For simplicity, we refer to these methods as CDFCG, CDFLBFGS, CDFTR, and CDFGD, respectively.

For comparison, we also employed the Riemannian gradient descent (RGD) \cite{absil2008optimization} and Riemannian conjugate gradient (RCG) \cite{sato2022riemannian} methods to solve the corresponding \ref{ocp} problem. Notably, to the best of our knowledge, the PyManopt package \cite{townsend2016pymanopt} does not currently support the three manifold structures considered in this study. The limited availability of manifold optimization solvers, combined with the flexibility of our algorithmic framework to incorporate a broad range of Euclidean unconstrained optimization methods, further underscores the promising applicability and extensibility of our approach. The RGD method determines the step size using alternating Barzilai–Borwein (BB) steps combined with a nonmonotonic line search strategy, as described in \cite{iannazzo2018riemannian}. The parameters for the nonmonotonic line search follow the default settings adopted in OPtM \cite{wen2013feasible}. For the RCG method, we use the conjugate parameter $ \beta_{RCG} = \min\{ \beta^{R-FR},\beta^{R-DY} \} $, see \cite{sato2022riemannian} for details. The Riemannian geometric tools are selected as follows: the Cayley transform is adopted as the retraction mapping for both the Symplectic Stiefel manifold and the indefinite Stiefel manifold. Vector transport is performed via orthogonal projection, which requires solving a Lyapunov equation at each iteration. We solve this equation using a direct method \cite{sorensen2003direct}, as recommended in \cite{van2024riemannian}. For the third-order tensor Stiefel manifold, we employ the retraction mapping based on the t-QR decomposition \cite{kilmer2013third}, and the vector transport is also implemented via orthogonal projection. To ensure a fair comparison, the same initial point is chosen for all methods. We terminate these methods when the maximum number of iterations exceeds $100000$ or the CPU time limit of $ 1800 $ seconds, and the gradient tolerance is set to two criteria, $10^{-5}$ and $10^{-9}$.

\subsection{Least squares matching problem for symplectic Stiefel manifold}
In this subsection, we integrate various unconstrained optimization methods into our proposed algorithmic framework to evaluate its adaptability to existing approaches. All comparisons and results are based on the least squares matching problem
\begin{equation*}
  \begin{aligned}
  &\min_{X\in\mathbb{R}^{2n \times 2p}}&&f(X) = {\rm tr}(X^{\top}AXN)\\
  &\mathrm{s.t.}&& X^{\top} J_{2n}X J_{2p}^{\top} = I_{2p},
  \end{aligned}
\end{equation*}
where $ J_{2n} = \left[ \begin{matrix}
                          0 & I_n \\
                          -I_n & 0
                        \end{matrix} \right] $, $ A \in \mathbb{R}^{2n \times 2n} $ and $ N \in \mathbb{R}^{2p \times 2p} $.
This problem arises from the least squares matching problem on matrix Lie groups by Brockett
\begin{equation*}
  f(X):={\rm tr}(X^{\top}AXN - 2BX^{\top}),
\end{equation*}
where $ A,N,B $ are given matrices.

In this test, the matrix $ A \in \mathbb{R}^{2n \times 2n} $ is randomly generated by $ A = U\Lambda U^{\top} $, where $ U \in \mathbb{R}^{2n \times 2n} $ is an orthogonal matrix and $ \Lambda \in \mathbb{R}^{2n \times 2n}, N \in \mathbb{R}^{2p \times 2p} $ are diagonal matrices with $ \Lambda_{ii}=\lambda_{i}, \lambda_1 \geq \cdots \geq \lambda_{2n} \geq 0 $ and $ N_{jj} = \mu_j, \mu_1 \geq \cdots \geq \mu_{2p} \geq 0 $ as diagonal elements, respectively. The parameter $ \lambda_{i}, i = 1,2, \cdots, 2n $ determine the rate of decay of the eigenvalues of the matrix $A$, and $\mu_j, j = 1,2,\cdots , 2p$ represent the weights between the eigenvalues of the matrix $A$. We then set the penalty parameter to $ \beta = 0.012 $ to construct the corresponding \ref{cdf} function

Tables \ref{Ta1} and \ref{Ta2} present the numerical performance of all solvers on least squares matching problems of varying sizes. In these tables, ``Fval'', ``Iter'', ``Grad'', ``Feas'', and ``CPU time'' denote the final function value $ f(X^*) $, number of iterations, $\| \operatorname{grad}f(X^*) \|$, $\| C(X^*) \|$, and total runtime in seconds, respectively. Here $ X^* $ denotes the iteration point obtained when these methods stop.

The results in Tables \ref{Ta1} and \ref{Ta2} show that all solvers integrated with the \ref{cdf} framework achieve function values comparable to Riemannian optimization methods, while maintaining high feasibility accuracy. This confirms the effectiveness of our algorithmic framework, which provides greater flexibility for solving \ref{ocp} via unconstrained optimization. In addition, the CDF framework matches the efficiency of Riemannian solvers across different problem scales. For larger-scale problems, our exact penalty method surpasses Riemannian solvers in numerical performance, with CDFLBFGS exhibiting particularly strong results.

We evaluated the computational efficiency of our framework by comparing CDFGD, CDFCG, RGD, and RCG in 100-iteration experiments (Figure \ref{fig:1}). The RGD method allocates over $90\%$ of its CPU time to retraction and orthogonal projection operations. In contrast, CDFGD and CDFCG avoid these geometric computations entirely, lowering per-iteration costs significantly.
The RCG method incurs even higher computational time due to the additional need to perform vector transport, which involves solving a Lyapunov equation of size $ 2p \times 2p $. This overhead becomes increasingly significant as the problem dimension $ p $ grows.

\begin{table}[H]
\centering
  \caption{The numerical results of the least squares matching problem with fixed $ 2p = 10 $. The matrices $A$ and $N$ in the problem are generated by setting $ \lambda_i = a^{1-i} + b $ and $ \mu_j = 0.1*e^{-\frac{i}{p}} $ , respectively, where $ a > 1 $ and $ b \in (0,2) $ are randomly generated.}
  \resizebox{\textwidth}{!}
  {\begin{tabular}{cc|ccccc|ccccc}
  	\toprule
\multirow{2}{*}{$(2n,2p)$}& \multirow{2}{*}{Solver}& \multicolumn{5}{|c|}{{\rm tol} = 1e-5}&\multicolumn{5}{c}{{\rm tol} = 1e-9} \\ \cmidrule(lr){3-7} \cmidrule(l){8-12}
  & \multicolumn{1}{c|}{} & Fval & Iter & Grad & Feas & CPU time & Fval & Iter & Grad & Feas & CPU time  \\ \midrule
\multirow{6}{*}{$(100,10) $} & CDFGD & 6.29e-03  & 265    & 9.63e-06     & 3.19e-15     & 0.15  & 6.29e-03  & 863    & 8.83e-10     & 7.21e-16     & 0.49 \\
&  CDFCG & 6.29e-03  & 291    & 3.07e-06     & 1.71e-15     & 0.47 & 6.29e-03  & 760    & 5.75e-10     & 6.79e-16     & 1.17  \\
&  CDFLBFGS & 6.29e-03  & 173    & 7.14e-06     & 1.58e-15     & {\bf 0.14} & 6.29e-03  & 492    & 1.90e-10     & 8.08e-16     & {\bf 0.40} \\
&  CDFTR & 6.29e-03  & 12    & 2.85e-06     & 1.10e-15     & 0.79 & 6.29e-03  & 19    & 1.74e-10     & 8.30e-16     & 1.92 \\ \cmidrule(l){2-12}
& RGD & 6.29e-03  & 179    & 9.11e-06     & 1.64e-14     & 0.55 & 6.29e-03  & 953    & 8.50e-10     & 3.92e-14     & 1.91  \\
& RCG & 6.29e-03  & 251    & 8.67e-06     & 1.48e-14     & 1.12 & 6.29e-03  & 1210    & 8.79e-10     & 6.67e-14     & 5.74 \\ \midrule
\multirow{6}{*}{$(500,10) $} & CDFGD & 1.01e-03  & 299    & 8.59e-06     & 3.64e-14     & {\bf 0.30}  & 1.01e-03  & 4208    & 9.82e-10     & 1.24e-15     & 3.09  \\
&  CDFCG & 1.01e-03  & 514    & 3.79e-06     & 2.69e-14     & 1.78 & 1.01e-03  & 2421    & 7.58e-10     & 8.60e-16     & 7.45 \\
&  CDFLBFGS & 1.01e-03  & 321    & 8.43e-06     & 8.08e-14     & 0.63 & 1.01e-03  & 1620    & 7.73e-11     & 9.21e-16     & {\bf 2.92} \\
&  CDFTR & 1.01e-03  & 22    & 2.69e-06     & 1.04e-14     & 1.04 & 1.01e-03  & 45    & 4.91e-10     & 1.75e-15     & 5.87 \\ \cmidrule(l){2-12}
& RGD & 1.01e-03  & 138    & 9.64e-06     & 1.80e-14     & 6.72 & 1.01e-03  & 1817    & 9.88e-10     & 1.76e-14     & 104.63 \\
& RCG & 1.01e-03  & 208    & 8.52e-06     & 2.55e-14     & 13.99 & 1.01e-03  & 2423    & 9.88e-10     & 1.76e-14     & 162.13 \\ \midrule
\multirow{6}{*}{$(1000,10) $} & CDFGD & 4.83e-04  & 327    & 9.59e-06     & 9.18e-14     & {\bf 0.45}  & 4.77e-04  & 13926    & 9.89e-10     & 1.65e-15     & 12.37 \\
&  CDFCG & 4.82e-04  & 553    & 7.88e-06     & 1.01e-14     & 2.58 & 4.77e-04  & 5093    & 9.32e-10     & 1.83e-15     & 21.49 \\
&  CDFLBFGS & 4.84e-04  & 384    & 9.68e-06     & 8.78e-14     & 1.14 & 4.77e-04  & 3034    & 1.38e-10     & 1.44e-15     & {\bf 8.65} \\
&  CDFTR & 4.82e-04  & 29    & 7.35e-06     & 5.55e-14     & 2.51 & 4.77e-04  & 69    & 8.20e-10     & 1.70e-15     & 32.47 \\ \cmidrule(l){2-12}
& RGD &  4.91e-04  & 110    & 9.95e-06     & 9.62e-14     & 9.63  & 4.77e-04  & 3868    & 9.21e-10     & 2.28e-14     & 293.33  \\
& RCG & 4.88e-04  & 137    & 9.76e-06     & 1.54e-14     & 14.47 & 4.77e-04  & 7282    & 9.43e-10     & 1.82e-14     & 816.34 \\
 \bottomrule
  \end{tabular}}
\label{Ta1}
\end{table}

\begin{table}[H]
\centering
  \caption{The numerical results of the least squares matching problem with fixed $ 2n = 1000 $.  The matrices $A$ and $N$ in the problem are generated by setting $ \lambda_i = a^{1-i} + b $ and $ \mu_j = 0.1*e^{-\frac{i}{p}} $ , respectively, where $ a > 1 $ and $ b \in (0,2) $ are randomly generated.}
  \resizebox{\textwidth}{!}
  {\begin{tabular}{cc|ccccc|ccccc}
  	\toprule
\multirow{2}{*}{$(2n,2p)$}& \multirow{2}{*}{Solver}& \multicolumn{5}{|c|}{{\rm tol} = 1e-5}&\multicolumn{5}{c}{{\rm tol} = 1e-9} \\ \cmidrule(lr){3-7} \cmidrule(l){8-12}
  & \multicolumn{1}{c|}{} & Fval & Iter & Grad & Feas & CPU time & Fval & Iter & Grad & Feas & CPU time  \\ \midrule
\multirow{6}{*}{$(1000,10) $} & CDFGD & 4.79e-04  & 249    & 9.79e-06     & 1.07e-14     & {\bf 0.35}  & 4.62e-04  & 13722    & 9.67e-10     & 1.27e-15     & 13.34  \\
&  CDFCG & 4.71e-04  & 390    & 4.74e-06     & 1.40e-14     & 1.86 & 4.62e-04  & 5627    & 9.14e-10     & 1.31e-15     & 24.25 \\
&  CDFLBFGS & 4.77e-04  & 293    & 9.34e-05     & 7.35e-14     & 0.89 & 4.62e-04  & 3046    & 1.18e-10     & 1.05e-15     & {\bf 8.51} \\
&  CDFTR & 4.62e-04  & 29    & 2.79e-06     & 8.28e-14     & 5.68 & 4.62e-04  & 49    & 1.66e-10     & 1.62e-15     & 17.14 \\ \cmidrule(l){2-12}
& RGD & 4.77e-04  & 125    & 9.98e-06     & 1.72e-14     & 9.56 & 4.62e-04  & 3970    & 9.86e-10     & 6.93e-14     & 297.96  \\
& RCG & 4.77e-04  & 144    & 9.35e-06     & 1.65e-14     & 14.89 & 4.62e-04  & 6623    & 9.23e-10     & 2.69e-14     & 684.82 \\ \midrule
\multirow{6}{*}{$(1000,50) $} & CDFGD & 5.16e-03  & 671    & 9.91e-06     & 4.26e-14     & {\bf 7.80}  & 5.13e-03  & 19181    & 9.93e-10     & 3.90e-15     & 230.73  \\
&  CDFCG & 5.16e-03  & 300    & 7.49e-06     & 6.90e-14     & 13.96 & 5.13e-03  & 12782    & 8.43e-10     & 3.88e-15     & 595.83 \\
&  CDFLBFGS & 5.16e-03  & 531    & 8.44e-06     & 2.90e-14     & 11.06 & 5.13e-03  & 8123    & 2.17e-10     & 2.31e-15     & {\bf 169.19} \\
&  CDFTR & 5.15e-03  & 35    & 4.86e-06     & 7.55e-14     & 56.23 & 5.13e-03  & 67    & 1.66e-10     & 1.62e-15     & 458.14 \\ \cmidrule(l){2-12}
& RGD & 5.16e-03  & 214    & 9.96e-06     & 7.16e-14     & 22.18 & 5.13e-03  & 13186    & 9.88e-10     & 1.76e-14     & 1366.63 \\
& RCG & 5.15e-03  & 317    & 8.53e-06     & 9.35e-14     & 51.80 & -  & -    & -     & -     & $>1800$  \\ \midrule
\multirow{6}{*}{$(1000,100) $} & CDFGD & 1.56e-02  & 820    & 9.87e-06     & 1.60e-13     & {\bf 15.12}  & 1.56e-02  & 43992    & 9.48e-10     & 6.65e-15     & 812.15 \\
&  CDFCG & 1.56e-02  & 401    & 9.49e-06     & 1.05e-13     & 31.99 & 1.56e-02  & 14773    & 9.49e-10     & 6.18e-15    & 1178.49 \\
&  CDFLBFGS & 1.56e-02  & 511    & 9.31e-06     & 3.66e-13     & 21.33 &  1.56e-02  & 13702    & 2.52e-10     & 4.08e-15     & {\bf 571.94} \\
&  CDFTR & 1.56e-02  & 25    & 9.75e-06     & 7.94e-13     & 71.52 & 1.56e-02  & 143    & 2.45e-10     & 4.76e-15     & 1547.47 \\ \cmidrule(l){2-12}
& RGD &  1.56e-02  & 292    & 9.32e-06     & 1.50e-13     & 85.36  & -  & -    & -     & -     & $>1800$   \\
& RCG & 1.56e-02  & 354    & 9.67e-06     & 1.76e-13     & 312.73 & -  & -    & -     & -     & $>1800$  \\
 \bottomrule
  \end{tabular}}
\label{Ta2}
\end{table}

\begin{figure}[H]
\centering
\begin{subfigure}{\textwidth}
	\includegraphics[width=0.75\textwidth]{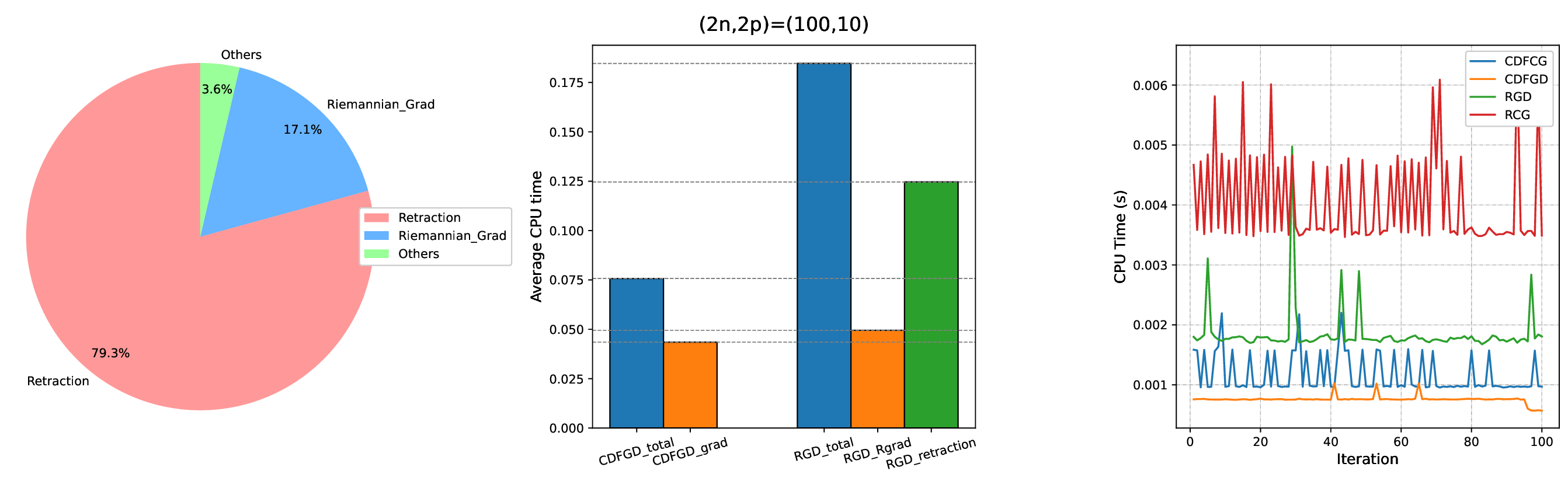}
\end{subfigure}
\begin{subfigure}{\textwidth}
	\includegraphics[width=0.75\textwidth]{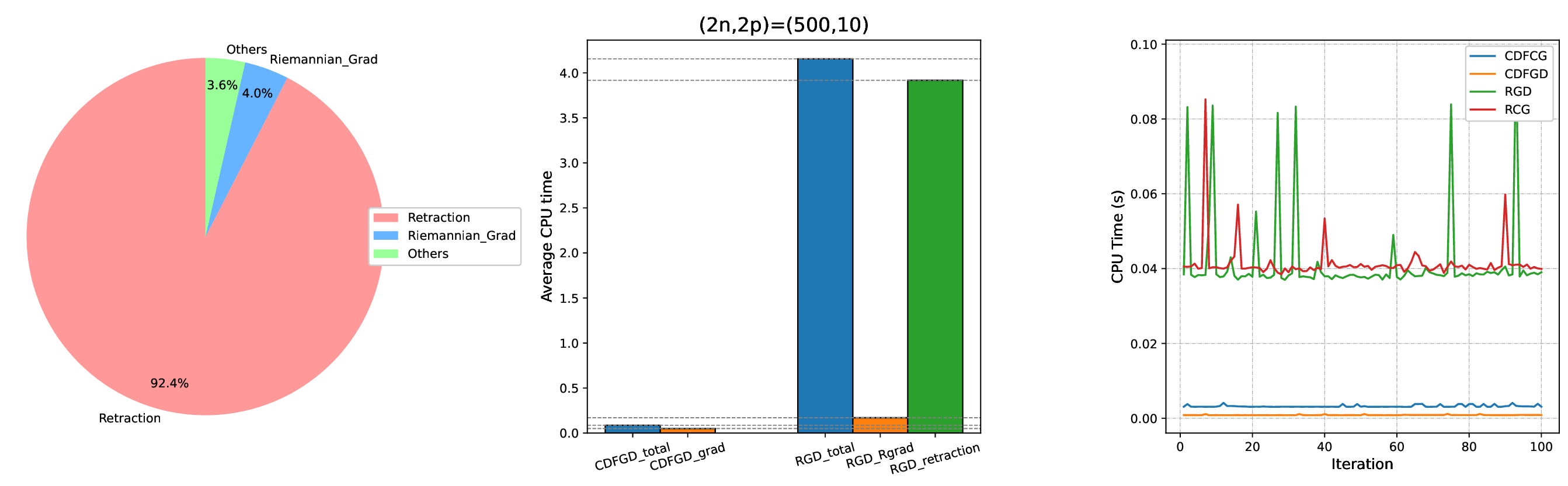}
\end{subfigure}
\begin{subfigure}{\textwidth}
	\includegraphics[width=0.75\textwidth]{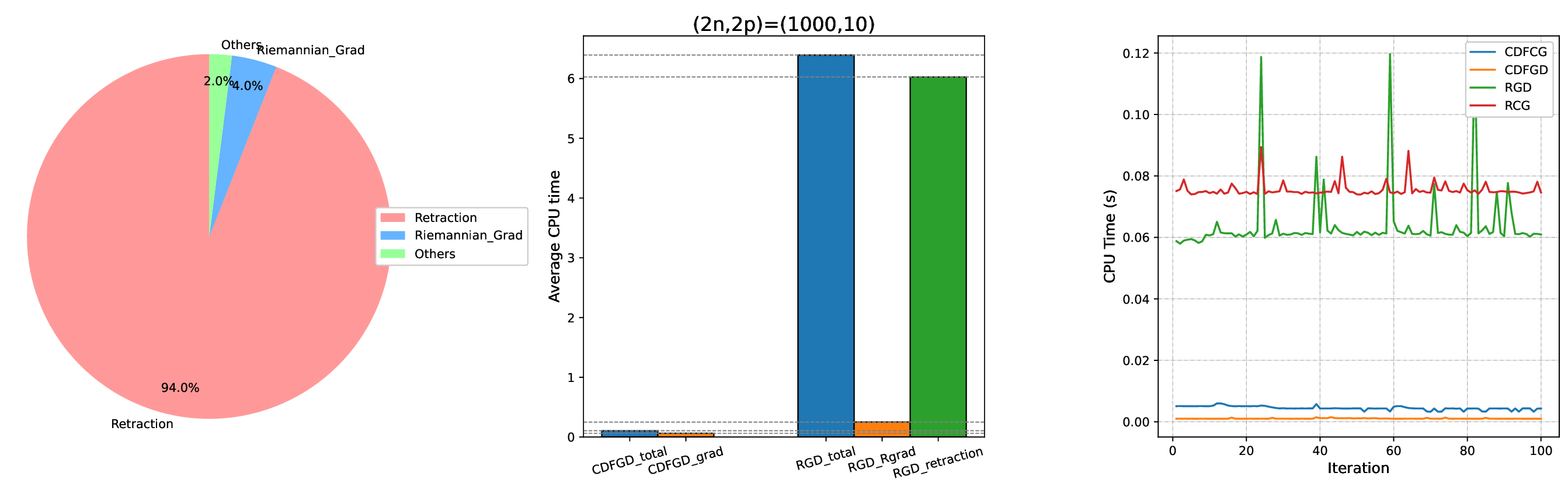}
\end{subfigure}
\begin{subfigure}{\textwidth}
	\includegraphics[width=0.75\textwidth]{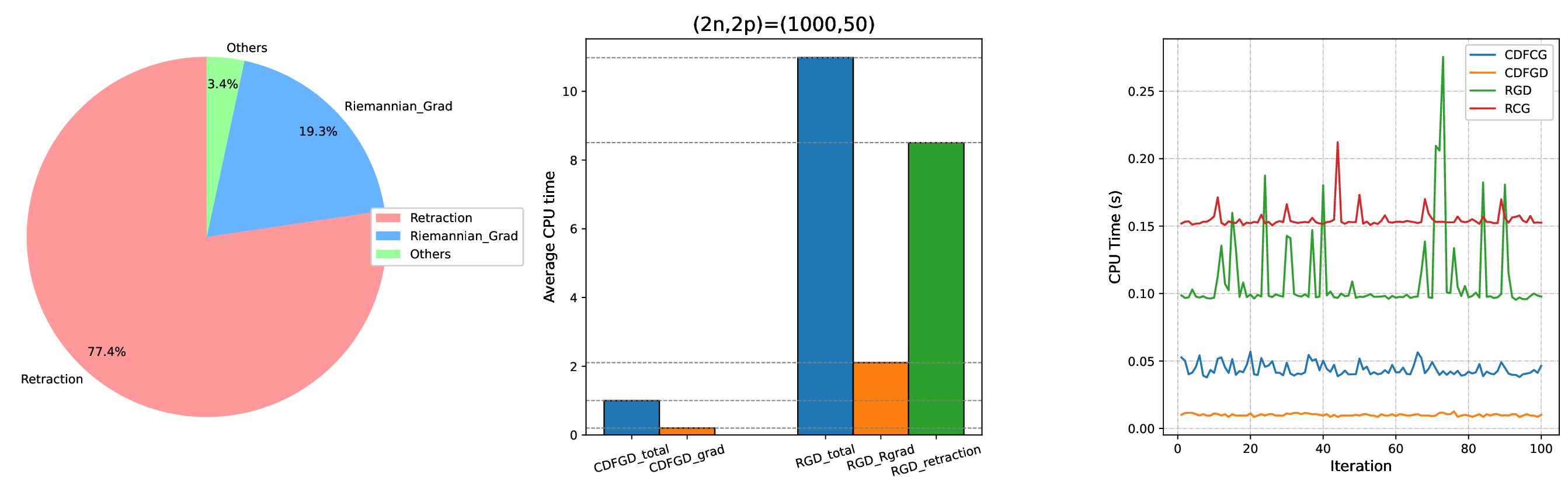}
\end{subfigure}
\begin{subfigure}{\textwidth}
	\includegraphics[width=0.75\textwidth]{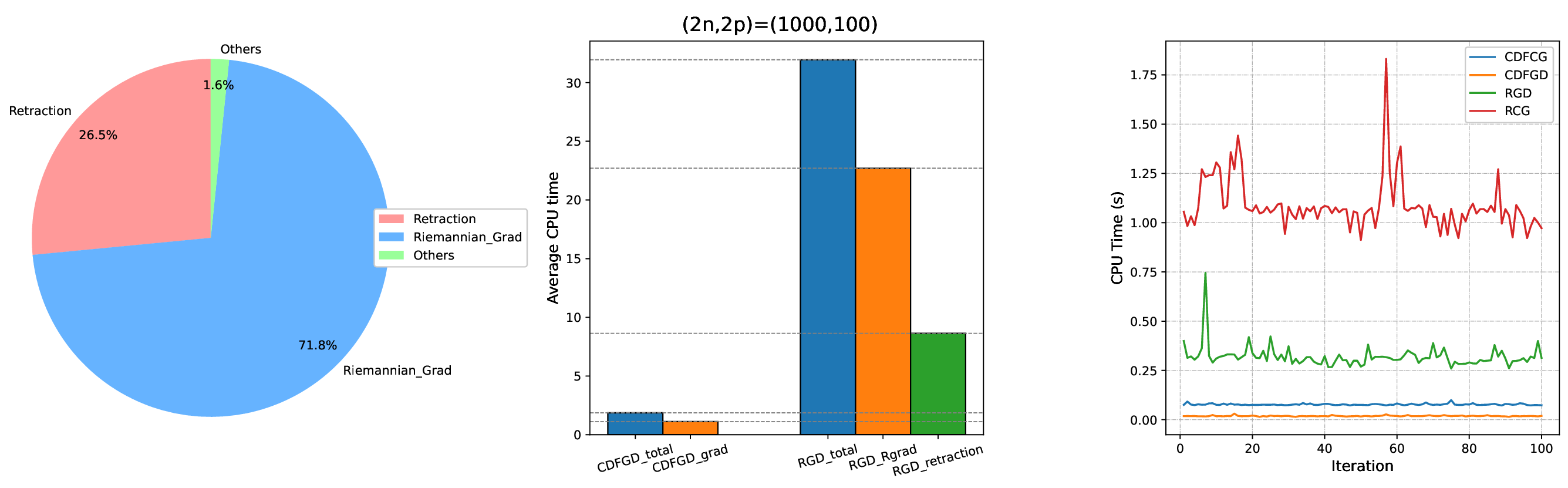}
\end{subfigure}
\caption{Computational time comparison for CDFGD, CDFCG, RGD, and RCG on least squares matching problems of varying sizes. Each method runs for 100 iterations. The left panel shows the percentage of time RGD spends computing the retraction and Riemannian gradient. The middle panel compares the total computation time, gradient computation time, and retraction computation time for CDFGD and RGD. For CDFGD, only the total and gradient computation times are shown, as it does not require retraction. The right panel presents the per-iteration computation time for all four methods.}
\label{fig:1}
\end{figure}


\subsection{Extrinsic mean problem for indefinite Stiefel manifold}
In this subsection we consider the extrinsic mean problem \cite{bhattacharya2003large} for indefinite Stiefel manifold
\begin{equation}\label{emp}
  \begin{aligned}
  &\min_{X\in\mathbb{R}^{n \times p}}&&\frac{1}{N}\sum_{i=1}^{N}\| X - X_i \|^2\\
  &\mathrm{s.t.}&&X^{\top} BXJ=I_p,
  \end{aligned}
\end{equation}
where all $ X_i $ satisfy $ X_{i}^{\top} BX_{i}J=I_p $ for $ i = 1, \cdots, N $, $B \in \mathbb{R}^{n \times n}$ is symmetric, nonsingular and $J \in \mathbb{R}^{p \times p}$ satisfying $ J^2 = I_p $. According to \cite{bhattacharya2003large}, problem \eqref{emp} can be described as a matrix approximation problem
\begin{equation*}\label{emp1}
  \begin{aligned}
  &\min_{X\in\mathbb{R}^{n \times p}}&&\| X - A \|^2\\
  &\mathrm{s.t.}&&X^{\top} BXJ=I_p,
  \end{aligned}
\end{equation*}
where $ A = \frac{1}{N}\sum_{i = 1}^{N} X_i $. In this experiment, we set $ B = \operatorname{diag}(1, \cdots, k, -m, \cdots, -1) $ with $ k+m = n $, $ J = \operatorname{diag}(I_{p_k}, I_{p_m}) $ and $  p_k + p_m = p $, $ p_k \leq k, p_m \leq m $.

In this experiment, all test cases are fixed to a sample size of $N = 1000$. Specifically, we randomly generate these samples around a central sample $ Y_0 $, i.e., $ X_i = Y_0 \left( \begin{matrix}
                                           W_1^{i} & 0 \\
                                           0 & W_2^{i}
                                         \end{matrix} \right) $, where $ Y_0^{\top} BY_0J=I_p $, $ W_1^i \in O_{p_k}$ and $ W_2^i \in O_{p_m} $. The penalty parameter is set to $ \beta = 0.5 $.
We first solve the problem \eqref{emp} using the CDFCG method to demonstrate the feasibility of the CDF framework for this problem. The scale is $ 1000 \times 100 $. We randomly select 100 samples and display both the initial residuals and the final residuals obtained by CDFCG in Figure \ref{fig:2}. As shown in the figure, the final solution produced by CDFCG method can effectively approximate all the samples, confirming the feasibility of the proposed approach.

To evaluate the numerical performance of all methods on this problem, we divide the tests into two groups:  (1) fixing $p$ while varying $ n $, and (2) fixing $ n $ while varying $ p $. Tables \ref{Ta3} and \ref{Ta4} report the numerical results of all solvers under two accuracy requirements. As observed from the tables, when the problem size is small, all solvers perform comparably. However, as the dimension increases (particularly $ p $ increases), the CDF-based methods begin to outperform the Riemannian optimization methods. Additionally, we record the time consumption for $100$ iterations of the CDFGD, CDFCG, RGD, and RCG methods in Figure \ref{fig:3}. The figure shows that with increasing $ p $, the time required to compute the Riemannian gradient grows rapidly, significantly impacting the efficiency of the Riemannian solvers. In particular, the RCG method becomes less efficient than RGD due to the additional cost of computing orthogonal projections. In contrast, the CDFGD and CDFCG methods remain consistently more efficient, both in terms of Euclidean gradient computation and overall per-iteration runtime.

\begin{figure}[H]
\centering
	\includegraphics[width=0.5\textwidth]{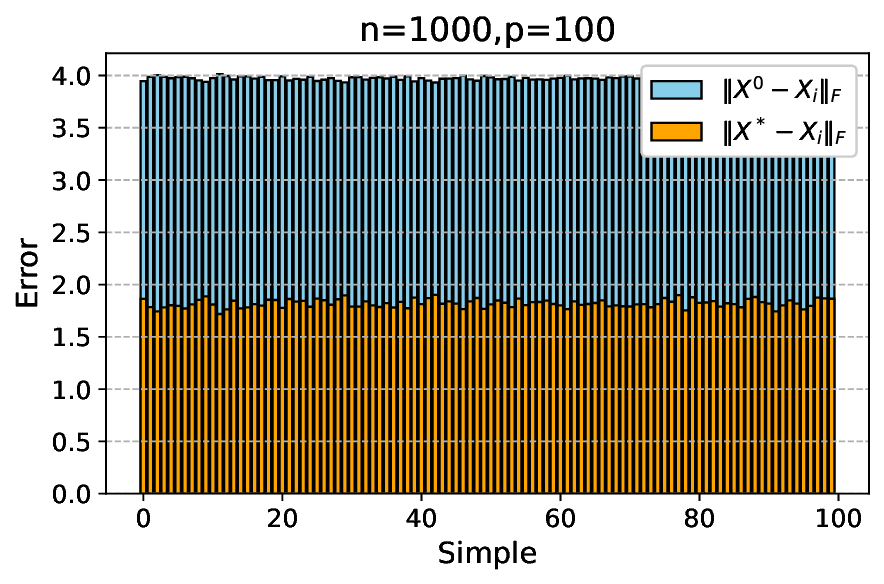}
\caption{$(n,p) = (1000,10)$. Comparison of initial and final errors of the CDFCG Method in solving the extrinsic mean problem for indefinite Stiefel manifold. A total of 100 randomly generated samples were taken. The blue bars represent the error between the initial iterate and each sample, while the orange bars depict the error between the convergence point and each sample.}
\label{fig:2}
\end{figure}

\begin{table}[H]
\centering
  \caption{The numerical results of the extrinsic mean problem with fixed $ p = 20 $. The number of samples $ N = 1000 $, and other parameters are set to $ k = 0.6n $ and $ p_k = 15 $.}
  \resizebox{\textwidth}{!}
  {\begin{tabular}{cc|ccccc|ccccc}
  	\toprule
\multirow{2}{*}{$(n,p)$}& \multirow{2}{*}{Solver}& \multicolumn{5}{|c|}{{\rm tol} = 1e-5}&\multicolumn{5}{c}{{\rm tol} = 1e-9} \\ \cmidrule(lr){3-7} \cmidrule(l){8-12}
  & \multicolumn{1}{c|}{} & Fval & Iter & Grad & Feas & CPU time & Fval & Iter & Grad & Feas & CPU time  \\ \midrule
\multirow{6}{*}{$(1000,20) $} & CDFGD & 3.74e-01  & 62    & 8.86e-06     & 1.01e-15     & {\bf 0.62}  & 3.74e-01  & 125    & 2.87e-10     & 1.02e-15     & {\bf 1.32} \\
&  CDFCG & 3.74e-01  & 40    & 8.61e-06     & 9.56e-15     & 0.79 & 3.74e-01  & 96    & 3.41e-10     & 1.27e-15     & 2.94  \\
&  CDFLBFGS & 3.74e-01  & 55    & 9.30e-06     & 1.32e-15     & 0.91 & 3.74e-01  & 145    & 3.87e-10     & 7.46e-16     & 2.31 \\
&  CDFTR & 3.74e-01  & 4    & 8.17e-05     & 7.86e-16     & 0.79 & 3.74e-01  & 6    & 1.88e-09     & 1.13e-15     & 1.65 \\ \cmidrule(l){2-12}
& RGD & 3.74e-01  & 5    & 9.11e-06     & 1.64e-14     & 0.83 & 3.74e-01  & 10    & 2.48e-10     & 5.34e-15     & 1.61 \\
& RCG & 3.74e-01  & 5    & 6.12e-06     & 4.92e-15     & 0.96 & 3.74e-01  & 11    & 2.87e-10     & 4.99e-15     & 1.92 \\ \midrule
\multirow{6}{*}{$(5000,20) $} & CDFGD & 3.70e-01  & 58    & 1.05e-06     & 1.18e-15     & 10.77 & 3.70e-01  & 153    & 8.58e-11     & 9.75e-16     & 24.43  \\
&  CDFCG & 3.70e-01  & 33    & 8.48e-06     & 8.17e-15     & 13.84  &  3.70e-01  & 121    & 2.30e-10     & 1.22e-15     & 47.37 \\
&  CDFLBFGS & 3.70e-01  & 56    & 6.84e-06     & 1.13e-15     & 12.84  & 3.70e-01  & 96    & 8.26e-10     & 1.07e-15     & 35.40 \\
&  CDFTR & 3.70e-01  & 3    & 5.53e-06     & 1.11e-15     & {\bf 6.57} & 3.70e-01  & 5    & 2.45e-10     & 1.25e-15     & {\bf 22.40} \\ \cmidrule(l){2-12}
& RGD & 3.70e-01  & 5    & 9.76e-06     & 4.37e-15     & 12.60 & 3.70e-01  & 10    & 2.21e-10     & 4.62e-15     & 25.15 \\
& RCG & 3.70e-01  & 6    & 7.02e-06     & 4.54e-15     & 16.80 & 3.70e-01  & 12    & 6.80e-10     & 5.33e-15     & 34.36 \\ \midrule
\multirow{6}{*}{$(10000,20) $} & CDFGD & 3.67e-01  & 94    & 3.44e-06     & 1.09e-15     & 35.10  & 3.67e-01  & 197    & 1.17e-10     & 1.35e-15     & {\bf 71.48} \\
&  CDFCG & 3.67e-01  & 86    & 4.11e-06     & 7.29e-16     & 69.68 & 3.67e-01  & 156    & 9.47e-10     & 7.86e-16     & 126.58 \\
&  CDFLBFGS &  3.67e-01  & 320    & 8.53e-06     & 1.18e-15     & 123.87 & 3.67e-01  & 595    & 1.09e-07     & 1.19e-15     & 229.63 \\
&  CDFTR & 3.67e-01  & 4    & 1.30e-04     & 2.27e-14     & {\bf 23.50} & 3.67e-01  & 20    & 5.99e-10     & 8.90e-16     & 434.35 \\ \cmidrule(l){2-12}
& RGD &  3.67e-01  & 5    & 3.42e-06     & 4.98e-15     & 61.66 & 3.67e-01  & 10    & 2.90e-10     & 6.38e-15     & 121.28 \\
& RCG &  3.67e-01  & 7    & 9.94e-06     & 3.85e-15     & 134.07 & 3.67e-01  & 13    & 8.02e-10     & 5.24e-15     & 171.90 \\
 \bottomrule
  \end{tabular}}
\label{Ta3}
\end{table}

\begin{table}[H]
\centering
  \caption{The numerical results of the extrinsic mean problem with fixed $ n = 1000 $. The number of samples $ N = 1000 $, and other parameters are set to $ k = 600 $ and $ p_k = 0.6p $.}
  \resizebox{\textwidth}{!}
  {\begin{tabular}{cc|ccccc|ccccc}
  	\toprule
\multirow{2}{*}{$(n,p)$}& \multirow{2}{*}{Solver}& \multicolumn{5}{|c|}{{\rm tol} = 1e-5}&\multicolumn{5}{c}{{\rm tol} = 1e-9} \\ \cmidrule(lr){3-7} \cmidrule(l){8-12}
  & \multicolumn{1}{c|}{} & Fval & Iter & Grad & Feas & CPU time & Fval & Iter & Grad & Feas & CPU time  \\ \midrule
\multirow{6}{*}{$(1000,10) $} & CDFGD & 2.37e-01  & 35    & 1.21e-06     & 7.37e-16     & {\bf 0.14}  & 2.37e-01  & 105    & 2.80e-11     & 9.42e-16     & {\bf 0.25} \\
&  CDFCG & 2.37e-01  & 26    & 3.71e-06     & 8.96e-16     & 0.16 & 2.37e-01  & 81    & 3.43e-10     & 5.55e-16     & 0.51  \\
&  CDFLBFGS & 2.37e-01  & 44    & 7.39e-06     & 5.21e-16     & 0.23  & 2.37e-01  & 157    & 5.74e-08     & 7.61e-16     & 0.62 \\
&  CDFTR & 2.37e-01  & 2    & 6.42e-06     & 3.53e-16     & 0.24 & 2.37e-01  & 4    & 1.82e-09     & 8.46e-16     & 0.36 \\ \cmidrule(l){2-12}
& RGD & 2.37e-01  & 5    & 5.92e-06     & 7.70e-16     & 0.40 & 2.37e-01  & 8    & 3.57e-11     & 1.29e-15     & 0.58  \\
& RCG & 2.37e-01  & 5    & 8.39e-06     & 1.72e-15     & 0.42 & 2.37e-01  & 10    & 6.28e-10     & 2.53e-15     & 0.90  \\ \midrule
\multirow{6}{*}{$(1000,500) $} & CDFGD & 9.42e-01  & 226    & 8.09e-06     & 1.07e-14     & 28.79  & 9.42e-01  & 542    & 7.36e-10     & 1.15e-14     & 69.24  \\
&  CDFCG & 9.42e-01  & 116    & 8.97e-06     & 1.08e-14     & 27.31 & 9.42e-01  & 288    & 9.44e-10     & 1.15e-14     & 60.97 \\
&  CDFLBFGS & 9.42e-01  & 85    & 9.85e-06     & 1.37e-14     & {\bf 19.87} &  9.42e-01  & 208    & 9.46e-10     & 1.15e-14     & {\bf 46.39} \\
&  CDFTR & 9.42e-01  & 4    & 4.26e-06     & 1.44e-14     & 18.49 & 9.42e-01  & 7    & 4.56e-13     & 1.15e-14     & 91.18 \\ \cmidrule(l){2-12}
& RGD & 9.42e-01  & 9    & 9.00e-06     & 7.66e-14     & 65.83 &  9.42e-01  & 16    & 9.43e-10     & 8.78e-14     & 114.27 \\
& RCG & 9.42e-01  & 6    & 5.71e-06     & 4.03e-14     & 107.32 & 9.42e-01  & 16    & 4.32e-10     & 7.12e-14     & 271.06 \\ \midrule
\multirow{6}{*}{$(1000,1000) $} & CDFGD & 4.18e+00  & 588    & 7.66e-06     & 1.87e-14     & 109.19  & 4.18e+00  & 1583    & 2.41e-10     & 1.81e-14     & {\bf 293.48} \\
&  CDFCG & 4.18e+00  & 794    & 2.03e-06     & 1.87e-14     & 234.83 & 4.18e+00  & 1283    & 5.11e-07     & 1.88e-14     & 374.64 \\
&  CDFLBFGS & 4.18e+00  & 236    & 9.55e-06     & 1.94e-14     & {\bf 92.31} & 4.18e+00  & 928    & 2.92e-07     & 1.84e-14     & 369.67 \\
&  CDFTR & 4.18e+00  & 5    & 5.38e-06     & 1.81e-14     & 96.53 & 4.18e+00  & 7    & 1.92e-10     & 1.82e-14     & 301.28 \\ \cmidrule(l){2-12}
& RGD &  4.18e+00  & 9    & 4.59e-06     & 8.27e-14     & 341.36  & 4.18e+00  & 18    & 9.00e-10     & 1.91e-13     & 696.28  \\
& RCG & 4.18e+00  & 8    & 5.96e-06     & 1.17e-13     & 697.88  & 4.18e+00  & 18    & 9.45e-10     & 1.45e-13     & 1596.28 \\
 \bottomrule
  \end{tabular}}
\label{Ta4}
\end{table}

\begin{figure}
\centering
\begin{subfigure}{\textwidth}
	\includegraphics[width=0.75\textwidth]{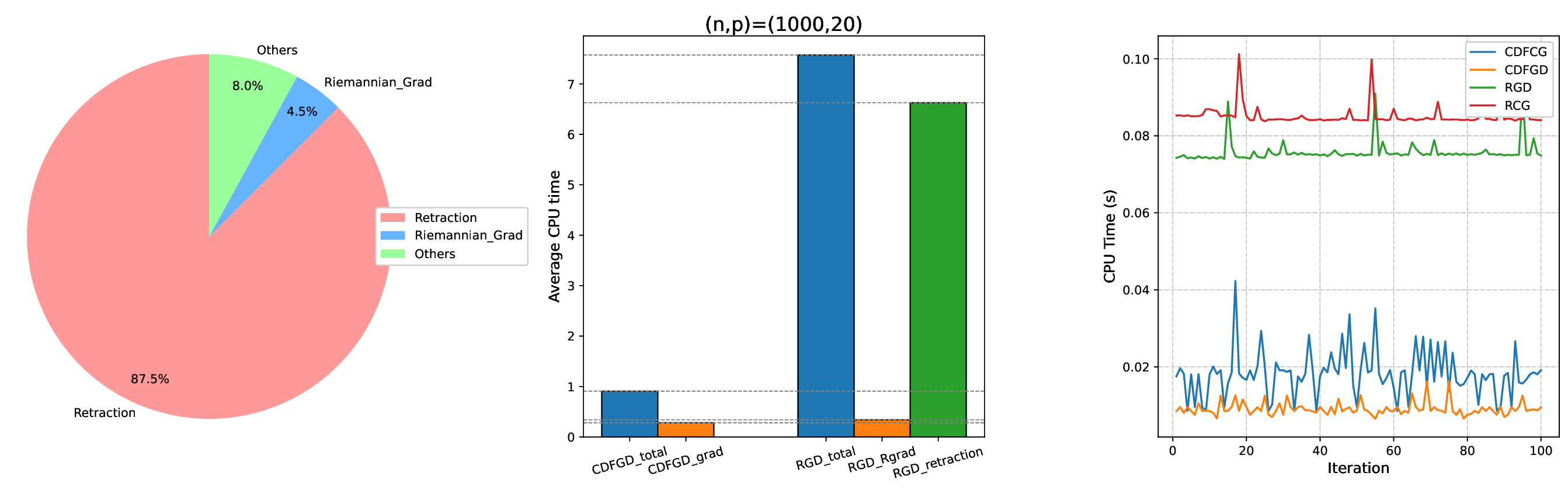}
\end{subfigure}
\begin{subfigure}{\textwidth}
	\includegraphics[width=0.75\textwidth]{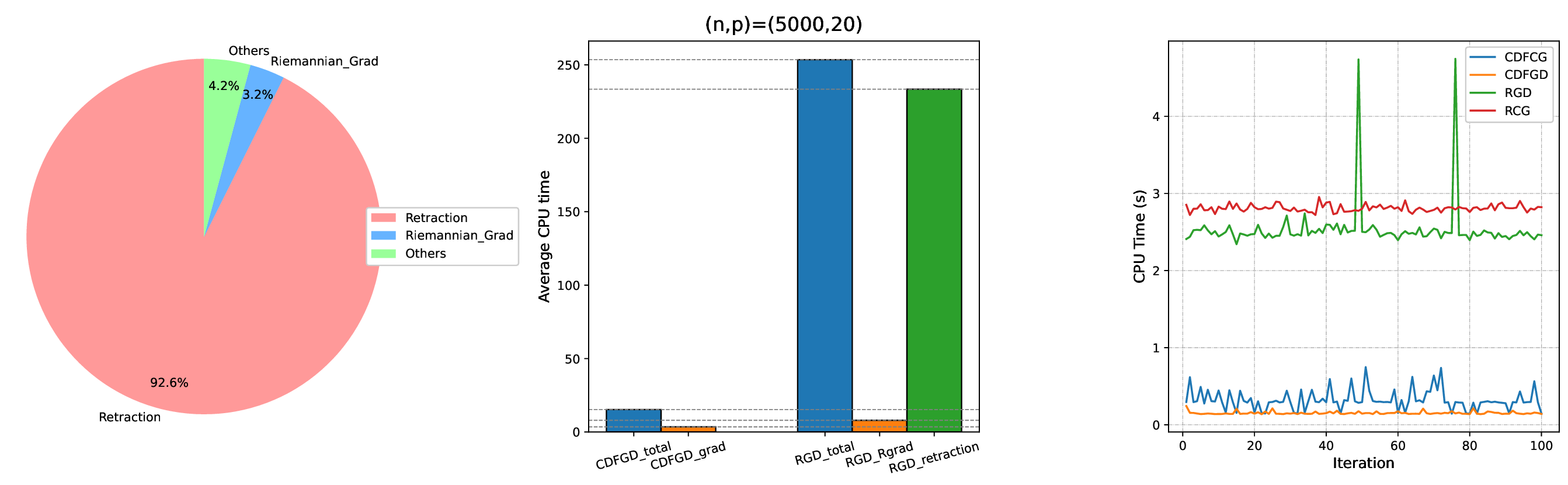}
\end{subfigure}
\begin{subfigure}{\textwidth}
	\includegraphics[width=0.75\textwidth]{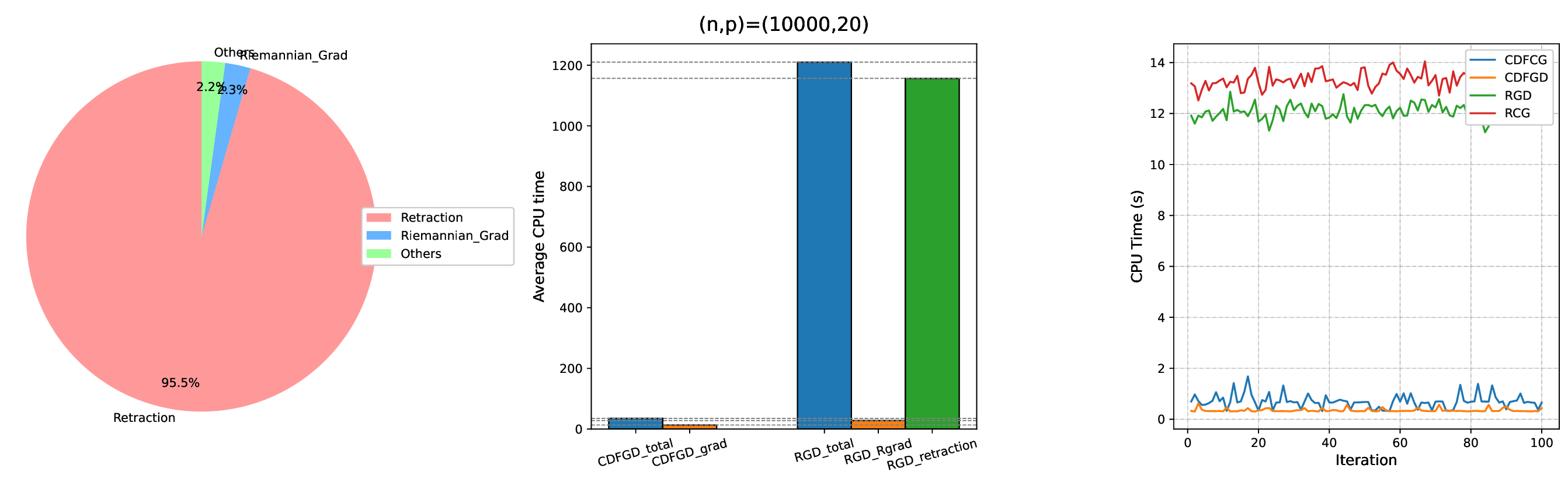}
\end{subfigure}
\begin{subfigure}{\textwidth}
	\includegraphics[width=0.75\textwidth]{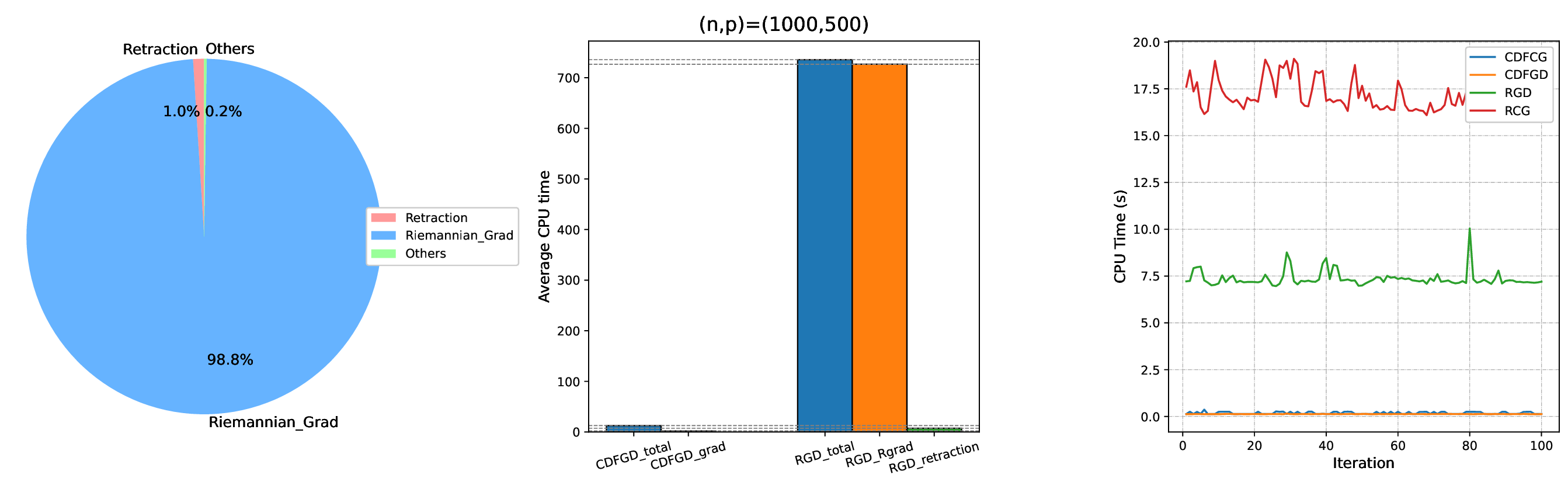}
\end{subfigure}
\begin{subfigure}{\textwidth}
	\includegraphics[width=0.75\textwidth]{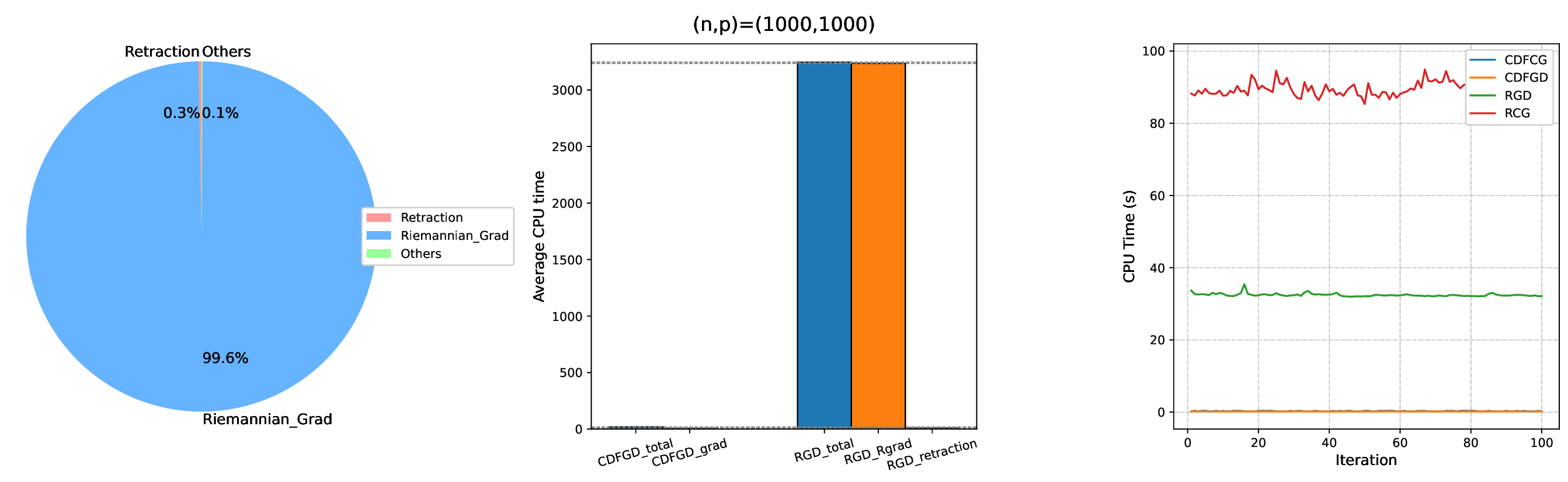}
\end{subfigure}
\caption{Computational time comparison for CDFGD, CDFCG, RGD, and RCG on extrinsic mean problems of varying sizes. Each method runs for 100 iterations. The left panel shows the percentage of time RGD spends computing the retraction and Riemannian gradient. The middle panel compares the total computation time, gradient computation time, and retraction computation time for CDFGD and RGD. For CDFGD, only the total and gradient computation times are shown, as it does not require retraction. The right panel presents the per-iteration computation time for all four methods.}
\label{fig:3}
\end{figure}

\subsection{Tensor joint f-diagonalization problem for third-order tensor Stiefel manifold}
In this subsection we consider the tensor joint f-diagonalization problem \cite{mao2024computation} for third-order tensor Stiefel manifold
\begin{equation}\label{tenpr1}
  \begin{aligned}
  &\min_{\mathscr{X}\in\mathbb{R}^{n \times p \times l}}&&f(\mathscr{X}) =  \sum_{i=i}^{N} \operatorname{off}\left( \mathscr{X}^{\top} *_c \mathscr{A}_i *_c \mathscr{X} \right) \\
  &\mathrm{s.t.}&&\mathscr{X}^{\top} *_c \mathscr{X}=\mathscr{I}_p,
  \end{aligned}
\end{equation}
where $ \mathscr{A}_i \in \mathbb{R}^{n \times n \times l} $ for $ i = 1, 2, \cdots, N $, $ \operatorname{off}(\mathscr{X}) = \sum_{i_3 = 1}^{l}\sum_{1\leq i_1 \neq i_2 \leq p}^{p} (x_{i_1i_2i_3})^2 $, $ \mathscr{I}_p \in \mathbb{R}^{p \times p \times l} $ and $ *_c $ denotes the cosine transform product \cite{kernfeld2015tensor}. The model is to seek a common third-order orthogonal tensor $ \mathscr{X} $ such that all sample tensors $ \{ \mathscr{A}_1, \cdots, \mathscr{A}_N \} $ are diagonalized as much as possible.

In this experiment, we set $ \mathscr{A}_i = \mathscr{U} *_c \mathscr{S}_i *_c \mathscr{U}^{\top} + \gamma\frac{\mathscr{E}_i}{\mathscr{E}_i}, i = 1, \cdots, N $, where $ \mathscr{U} \in\mathbb{R}^{n \times p \times l} $ is a orthogonal tensor, $ \mathscr{S}_i \in \mathbb{R}^{p \times p \times l} $ are randomly generated f-diagonal tensors, and $ \mathscr{E}_i \in \mathbb{R}^{n \times n \times l} $ are randomly generated noises with noise level $ \gamma $. The penalty parameter is set to $ \beta = 0.8 $. We conduct two sets of experiments to evaluate the numerical performance of evaluate the numerical performance of CDFGD, CDFCG, CDFLBFGS, CDFTR, RGD and RCG methods. In the first set, we fix $  (n,p,l) $, and vary the number of samples $N$. In the second set, we fix $N$ and vary the tensor size, which is further divided into two subgroups: (i) fixing $p,l$ and varying $n$, and (ii) fixing $n,l$ and varying $p$.

Tables \ref{Ta5}-\ref{Ta7} report the performance of each solver in solving problem \eqref{tenpr1}. As observed from the results, all compared solvers achieved sufficiently small objective function values, indicating that all solvers diagonalized each sample tensor as much as possible. Although retraction and vector transport are explicitly defined in this experiment for all Riemannian optimization methods, the CDFCG and CDFGD methods still achieve numerical performance comparable to that of the RCG and RGD methods. It is worth noting that the performance of the CDFLBFGS method consistently has a significant advantage over the Riemannian solvers in all tests.

\begin{table}[H]
\centering
  \caption{The numerical results for the tensor joint f-diagonalization problem with fixed $(n,p,l)=(100,5,10)$. The noise level is set at $ \gamma = 0.5 $.}
  \resizebox{\textwidth}{!}
  {\begin{tabular}{cc|ccccc|ccccc}
  	\toprule
\multirow{2}{*}{$N$}& \multirow{2}{*}{Solver}& \multicolumn{5}{|c|}{{\rm tol} = 1e-5}&\multicolumn{5}{c}{{\rm tol} = 1e-9} \\ \cmidrule(lr){3-7} \cmidrule(l){8-12}
  & \multicolumn{1}{c|}{} & Fval & Iter & Grad & Feas & CPU time & Fval & Iter & Grad & Feas & CPU time  \\ \midrule
\multirow{6}{*}{$N = 10$} & CDFGD &  4.54e-11 & 223    & 7.58e-06     & 2.03e-15     & 8.32 &  1.18e-18 & 563    & 8.71e-10     & 1.67e-15     & 21.02  \\
&  CDFCG & 6.82e-12 & 254    & 3.89e-06     & 1.76e-15     & 11.36 & 6.37e-19 & 562    & 7.19e-10     & 1.89e-15     & 25.25  \\
&  CDFLBFGS & 2.14e-11  & 99    & 9.04e-06     & 1.90e-15     & {\bf 3.19}  &2.70e-23  & 333    & 1.11e-11     & 1.90e-15     & {\bf 10.62} \\
&  CDFTR & 2.59e-13 & 3    & 7.79e-07     & 1.58e-15     & 7.51 & 1.00e-19 & 5    & 4.43e-10     & 1.63e-15     & 23.72 \\ \cmidrule(l){2-12}
& RGD & 3.96e-11  & 373     & 8.19e-06     & 2.38e-15     & 13.23 & 2.49e-18  & 578    & 9.71e-10     & 2.67e-15     & 20.41 \\
& RCG & 1.65e-10  & 398    & 9.81e-06     & 2.32e-15     & 14.19 &  1.21e-18  & 912    & 9.05e-10     & 2.19e-15     & 34.23 \\ \midrule
\multirow{6}{*}{$N = 50 $} & CDFGD & 2.91e-11 & 277    & 9.60e-06     & 1.84e-15     & 47.11  & 1.32e-21 & 619    & 4.49e-10     & 1.86e-15     & 104.86   \\
&  CDFCG &  3.31e-12 & 234    & 4.56e-06     & 1.79e-15     & 48.15 & 5.03e-21 & 571    & 8.05e-10     & 1.84e-15     & 116.87 \\
&  CDFLBFGS & 4.81e-12  & 128    & 9.08e-06     & 1.51e-15     & {\bf 19.32} & 1.11e-24  & 452    & 8.67e-12     & 1.86e-15     & {\bf 64.92} \\
&  CDFTR &  5.02e-14 & 3    & 6.94e-07     & 1.79e-15     & 58.76 & 4.17e-20 & 4    & 3.72e-10     & 1.96e-15     & 112.18 \\ \cmidrule(l){2-12}
& RGD & 1.99e-11  & 313    & 6.97e-06     & 2.61e-15     & 52.64 & 1.31e-18  & 867    & 8.53e-10     & 2.41e-15     & 146.85 \\
& RCG & 1.30e-10  & 422    & 9.50e-06     & 2.56e-15     & 75.34 & 3.36e-19  & 975    & 8.26e-10     & 2.75e-15     & 174.57 \\ \midrule
\multirow{6}{*}{$N = 100$} & CDFGD & 1.85e-11 & 316    & 9.00e-06     & 1.67e-15     & 106.15  &  8.20e-24 & 743    & 9.95e-11     & 1.93e-15     & 247.16 \\
&  CDFCG & 3.77e-13 & 269    & 3.23e-06     & 1.77e-15     & 112.88 & 2.14e-21 & 738    & 8.54e-10     & 1.66e-15     & 303.68 \\
&  CDFLBFGS & 1.51e-12  & 96    & 8.85e-06     & 2.02e-15     & {\bf 27.49} & 1.83e-24  & 388    & 1.10e-11     & 1.73e-15     & {\bf 110.27} \\
&  CDFTR & 1.38e-13 & 3    & 2.52e-06     & 1.79e-15     & 74.34 & 2.93e-24 & 5    & 4.84e-12     & 1.54e-15     & 261.34 \\ \cmidrule(l){2-12}
& RGD &  6.36e-11  & 313    & 9.03e-06     & 2.49e-15     & 108.03  & 3.65e-19  & 798    & 8.95e-10     & 1.99e-15     & 269.37   \\
& RCG &  2.93e-11  & 463    & 9.76e-06     & 2.69e-15     & 165.08  & 4.36e-19  & 814    & 9.66e-10     & 2.66e-15     & 291.72 \\
 \bottomrule
  \end{tabular}}
\label{Ta5}
\end{table}

\begin{table}[H]
\centering
  \caption{The numerical results for the tensor joint f-diagonalization problem with fixed $N=30$ and $ p = 5 $. The noise level is set at $ \gamma = 0.8 $.}
  \resizebox{\textwidth}{!}
  {\begin{tabular}{cc|ccccc|ccccc}
  	\toprule
\multirow{2}{*}{$(n,p,l)$}& \multirow{2}{*}{Solver}& \multicolumn{5}{|c|}{{\rm tol} = 1e-5}&\multicolumn{5}{c}{{\rm tol} = 1e-9} \\ \cmidrule(lr){3-7} \cmidrule(l){8-12}
  & \multicolumn{1}{c|}{} & Fval & Iter & Grad & Feas & CPU time & Fval & Iter & Grad & Feas & CPU time  \\ \midrule
\multirow{6}{*}{$(100,5,5) $} & CDFGD &9.00e-11 & 128    & 9.73e-06     & 1.18e-15     & 6.59  &   1.35e-21 & 356    & 1.47e-10     & 1.30e-15     & 18.20  \\
&  CDFCG & 1.21e-12 & 126    & 2.61e-06     & 1.46e-15     & 7.79 & 1.27e-20 & 302    & 6.92e-10     & 1.26e-15     & 18.70  \\
&  CDFLBFGS &  1.07e-11  & 65    & 8.51e-06     & 1.40e-15     & {\bf 2.90}   &  7.09e-24  & 218    & 7.11e-12     & 1.05e-15     & {\bf 9.69}  \\
&  CDFTR & 9.57e-12 & 2    & 2.59e-06     & 9.00e-16     & 5.40 & 1.72e-22 & 4    & 2.28e-11     & 1.11e-15     & 17.47 \\ \cmidrule(l){2-12}
& RGD &   3.21e-11  & 161    & 8.32e-06     & 1.65e-15     & 8.15 & 3.38e-19  & 315    & 5.96e-10     & 1.42e-15     & 16.43 \\
& RCG &  1.95e-11  & 177    & 8.16e-06     & 1.61e-15     & 9.48  &  1.16e-19  & 349    & 9.04e-10     & 1.94e-15     & 19.16  \\ \midrule
\multirow{6}{*}{$(300,5,5) $} & CDFGD & 5.85e-11 & 221    & 7.88e-06     & 1.95e-15     & 121.19  & 2.44e-18 & 442    & 9.96e-10     & 2.30e-15     & 242.76  \\
&  CDFCG &  1.48e-11 & 178    & 3.38e-06     & 2.27e-15     & 119.58 & 1.79e-18 & 371    & 9.70e-10     & 2.26e-15     & 243.45 \\
&  CDFLBFGS &  1.44e-11  & 72    & 9.58e-06     & 2.38e-15     & {\bf 34.75} & 6.36e-23  & 263    & 1.47e-11     & 2.16e-15     & {\bf 119.89} \\
&  CDFTR &  1.16e-11 & 3    & 4.12e-06     & 2.35e-15     & 96.08  &  8.92e-26 & 5    & 2.50e-13     & 2.26e-15     & 259.67 \\ \cmidrule(l){2-12}
& RGD &  1.35e-11  & 227    & 7.33e-06     & 2.37e-15     & 123.21 & 2.68e-18  & 457    & 9.95e-10     & 2.40e-15     & 247.90 \\
& RCG & 1.05e-10  & 244    & 9.99e-06     & 2.29e-15     & 141.20 & 1.10e-18  & 501    & 8.68e-10     & 2.24e-15     & 287.91 \\ \midrule
\multirow{6}{*}{$(500,5,5) $} & CDFGD &  3.11e-09 & 339    & 9.93e-06     & 3.12e-15     & 566.07  & 3.11e-20 & 749    & 4.46e-10     & 3.53e-15     & 1251.37  \\
&  CDFCG &  3.36e-10 & 236    & 3.74e-06     & 2.35e-15     & 441.84 & 3.23e-19 & 655    & 9.63e-10     & 2.41e-15     & 1224.25 \\
&  CDFLBFGS & 4.46e-10  & 120    & 8.53e-06     & 2.80e-15     & {\bf 147.41} &  6.36e-22  & 415    & 1.73e-11     & 2.92e-15     & {\bf 524.06} \\
&  CDFTR & 7.97e-10 & 5    & 3.76e-06     & 2.56e-15     & 503.75  &  6.69e-22 & 7    & 2.90e-11     & 2.54e-15     & 892.65 \\ \cmidrule(l){2-12}
& RGD &    9.17e-10  & 343    & 8.49e-06     & 2.09e-15     & 583.64  &  3.82e-18  & 731    & 8.20e-10     & 3.04e-15     & 1242.91  \\
& RCG & 3.08e-10  & 452    & 9.74e-06     & 2.60e-15     & 818.62  &  -  & -    & -     & -     & $>1800$ \\
 \bottomrule
  \end{tabular}}
\label{Ta6}
\end{table}

\begin{table}[H]
\centering
  \caption{The numerical results for the tensor joint f-diagonalization problem with fixed $N=30$ and $ n = 100 $. The noise level is set at $ \gamma = 0.8 $.}
  \resizebox{\textwidth}{!}
  {\begin{tabular}{cc|ccccc|ccccc}
  	\toprule
\multirow{2}{*}{$(n,p,l)$}& \multirow{2}{*}{Solver}& \multicolumn{5}{|c|}{{\rm tol} = 1e-5}&\multicolumn{5}{c}{{\rm tol} = 1e-9} \\ \cmidrule(lr){3-7} \cmidrule(l){8-12}
  & \multicolumn{1}{c|}{} & Fval & Iter & Grad & Feas & CPU time & Fval & Iter & Grad & Feas & CPU time  \\ \midrule
\multirow{6}{*}{$(100,5,5) $} & CDFGD &  6.14e-11 & 184    & 9.89e-06     & 1.14e-15     & 9.46  & 8.80e-20 & 453    & 7.66e-10     & 1.26e-15     & 23.13 \\
&  CDFCG &  2.10e-12 & 225    & 2.74e-06     & 1.28e-15     & 14.17 & 4.31e-20 & 452    & 5.01e-10     & 1.36e-15     & 28.02  \\
&  CDFLBFGS & 1.18e-11  & 105    & 8.01e-06     & 1.46e-15     & {\bf 4.64} & 1.00e-23  & 297    & 9.84e-12     & 1.29e-15     & {\bf 12.83} \\
&  CDFTR & 1.11e-11 & 3    & 1.37e-06     & 1.20e-15     & 8.52 & 1.77e-23 & 5    & 1.01e-11     & 1.21e-15     & 22.59 \\ \cmidrule(l){2-12}
& RGD & 3.50e-11  & 222    & 7.03e-06     & 2.06e-15     & 11.20 &  2.21e-19  & 513    & 6.46e-10     & 1.75e-15     & 25.98  \\
& RCG &  4.87e-11  & 407    & 9.48e-06     & 2.11e-15     & 21.77 & 4.60e-19  & 723    & 9.90e-10     & 1.77e-15     & 38.13  \\ \midrule
\multirow{6}{*}{$(100,10,5) $} & CDFGD &  1.51e-10 & 306    & 9.71e-06     & 1.72e-15     & 30.18  & 4.95e-21 & 629    & 6.02e-10     & 2.08e-15     & 61.94 \\
&  CDFCG &  4.84e-12 & 266    & 4.05e-06     & 2.04e-15     & 31.44 & 3.33e-21 & 669    & 7.22e-10     & 2.08e-15     & 78.65  \\
&  CDFLBFGS & 1.10e-11  & 146    & 8.04e-06     & 2.09e-15     & {\bf 12.19} & 1.19e-23  & 507    & 1.06e-11     & 1.86e-15     & {\bf 41.96} \\
&  CDFTR & 2.96e-11 & 3    & 2.24e-06     & 1.92e-15     & 27.43 & 4.97e-23 & 5    & 3.07e-11     & 1.92e-15     & 91.58 \\ \cmidrule(l){2-12}
& RGD & 1.04e-10  & 370    & 8.29e-06     & 2.40e-15     & 36.87 &  1.69e-18  & 628    & 9.75e-10     & 2.36e-15     & 62.67  \\
& RCG & 1.04e-10  & 524    & 9.68e-06     & 2.47e-15     & 55.20 & 6.29e-19  & 810    & 8.94e-10     & 3.62e-14     & 85.74 \\ \midrule
\multirow{6}{*}{$(100,30,5) $} & CDFGD &  2.73e-10 & 626    & 7.18e-06     & 3.91e-15     & 206.17  &  4.02e-24 & 2466    & 1.74e-11     & 3.63e-15     & 813.80  \\
&  CDFCG & 9.43e-11 & 625    & 9.41e-06     & 4.11e-15     & 243.67 & 2.71e-21 & 2578    & 8.14e-10     & 3.76e-15     & 1002.93 \\
&  CDFLBFGS &  1.85e-11  & 213    & 8.29e-06     & 3.80e-15     & {\bf 58.35}  &  2.14e-23  & 839    & 1.76e-11     & 3.86e-15     & {\bf 225.95}  \\
&  CDFTR & 6.38e-12 & 3    & 9.51e-06     & 4.00e-15     & 264.16 & 2.85e-20 & 6    & 6.56e-10     & 3.89e-15     & 1093.04  \\ \cmidrule(l){2-12}
& RGD & 4.10e-10  & 698    & 9.83e-06     & 4.93e-15     & 240.68 & 6.30e-18  & 3236    & 9.07e-10     & 5.45e-15     & 1108.84\\
& RCG &  3.49e-10  & 1035    & 8.55e-06     & 4.81e-15     & 381.48 & -  & -    & -     & -     & $>1800$ \\
 \bottomrule
  \end{tabular}}
\label{Ta7}
\end{table}

\section{Conclusion}
Optimization problems with generalized orthogonal constraints comprise a class of manifold optimization problems with special structural properties. Existing approaches have primarily relied on the framework outlined in \cite{absil2008optimization}, in which unconstrained optimization methods are extended to specific manifold structures and corresponding theoretical analyses are developed. However, these methods are typically designed on a case-by-case basis, which limits their extensibility and makes them difficult to apply directly to GOOCP. In addition, operations such as retraction and vector transport are typically involved in these methods, and their computational efficiency is often limited.

This paper shows that the feasible region of \ref{ocp} is a closed embedded submanifold of $\mathbb{R}^{n\times p}$. Consequently, \ref{ocp} can be reformulated as a class of Riemannian optimization problems. To facilitate the development of Riemannian optimization methods, explicit expressions for the tangent space, the Riemannian gradient, and the Riemannian Hessian of the manifold $\mathcal{M}$ are provided. On the other hand, based on the Riemannian constraint dissolving framework, we propose a computationally efficient constraint dissolving operator and subsequently construct the corresponding constraint dissolving function \ref{cdf} for \ref{ocp}. Theoretical properties of the exact penalty function are investigated, including the establishment of a threshold condition for the penalty parameter under which \ref{ocp} and \ref{cdf} share first-order or second-order stationary points within a neighborhood of the manifold, and an analysis is provided for the computational complexity associated with computing the gradient of \ref{cdf}. Through this framework, not only can the desirable properties of various unconstrained optimization methods in Euclidean space be inherited, but the computational burden associated with complex Riemannian geometric tools can also be effectively avoided.

In summary, \ref{cdf} bridges the gap between \ref{ocp} and unconstrained optimization, allowing the problem \ref{ocp} to be solved using various Euclidean versions of unconstrained optimization methods. The need to analyze the manifold structure of \ref{ocp} on a case-by-case basis is thereby eliminated, and a unified, scalable algorithmic framework is established. Furthermore, geometric operations such as retraction and vector transport are avoided, resulting in improved efficiency for solving large-scale problems. The potential advantages of Algorithm \ref{al:1} over traditional Riemannian optimization methods are supported by extensive numerical experiments.

\section*{Acknowledgements}
This work was supported by the National Natural Science Foundation of China (No. 12271113), and Guangxi Natural Science Foundation
(No. 2026GXNSFDA00640024).




\bibliographystyle{elsarticle-num}
  \bibliography{references}

\end{document}